\documentclass[10pt,leqno,twoside]{amsart}
\usepackage{enumerate}
\usepackage{amssymb}
\usepackage{amssymb,amsthm,amsmath,eucal,mathrsfs}

\usepackage{tikz, subfigure, xcolor} 
\usepackage{pgfplots}
\usepackage{cite}
\usepackage{amsmath}
\usepackage{amsthm}
\usepackage{amssymb}
\usepackage{amsfonts}
\usepackage{doi}
\usepackage{comment}

\usepackage{dsfont}
\usepackage{hyperref}

\newtheorem{theorem}{Theorem}[section]
\newtheorem{lemma}[theorem]{Lemma}
\newtheorem{proposition}[theorem]{Proposition}
\newtheorem{corollary}[theorem]{Corollary}

\newtheorem{claim}[theorem]{Claim}

\theoremstyle{definition}

\newtheorem{remark}[theorem]{Remark}
\newtheorem{remarks}[theorem]{Remarks}

\numberwithin{equation}{section}

\newcommand{\supp}{\mathrm{supp}}      
\renewcommand{\div}{\mathrm{div}\,}    

\providecommand{\norm}[1]{\lVert#1\rVert} 

\newcommand{\R}{\mathbb{R}}
\newcommand{\K}{\mathbb{K}}
\newcommand{\T}{\mathbb{T}}
\newcommand{\C}{\mathbb{C}}

\newcommand{\N}{\mathbb{N}}

\newcommand{\PP}{\mathbb{P}}

\newcommand{\cR}{{\mathcal R}}

\newcommand{\lpso}{L^p_{\overline{\sigma}}(\Omega)}

\newcommand{\os}{{\overline{\sigma}}}

\newcommand{\Ae}{\mathcal{A}} 

\newcommand{\Aip}{A}

\newcommand{\Aipos}{A_\os}

\newcommand{\Apos}{A_{p,\os}}

\newcommand{\Ap}{A_{p}}

\newcommand{\Xip}{X}

\newcommand{\Xipos}{X_\os}

\newcommand\restr[2]{{
  \left.\kern-\nulldelimiterspace 
  #1 
  \vphantom{\big|} 
  \right|_{#2} 
  }}

\usepackage{eucal}

\title[Hydrostatic Stokes Semigroup on Spaces of Bounded Functions]{The Hydrostatic Stokes Semigroup and Well-Posedness of the Primitive Equations on Spaces of Bounded Functions}

\subjclass[2010]{Primary: 35Q35; Secondary: 47D06, 76D03, 86A05.}

\keywords{hydrostatic Stokes semigroup, $L^\infty$-estimates,  primitive equations, global strong well-posedness \\
	This work was partly supported by the DFG International Research Training Group IRTG 1529 and the JSPS Japanese-German Graduate Externship on Mathematical Fluid Dynamics. 
	The first author is partly supported by JSPS through grant Kiban S (No. 26220702),
Kiban A (No. 17H01091), Kiban B (No. 16H03948), the second and fourth author are
supported by IRTG 1529 at TU Darmstadt, the fifth author is supported by JSPS Grant-in-Aid for Young Scientists B (No. 17K14230).}

\author[Giga]{Yoshikazu Giga} 
\address{Graduate School of Mathematical Sciences, University of Tokyo, Komaba 3-8-1, Meguro-ku, Tokyo, 153-8914, Japan }
\email{labgiga@ms.u-tokyo.ac.jp}

\author[Gries]{Mathis Gries} 
\address{Departement of Mathematics,
	TU Darmstadt, Schlossgartenstr. 7, 64289 Darmstadt, Germany}
\email{gries@mathematik.tu-darmstadt.de}

\author[Hieber]{Matthias Hieber} 
\address{Departement of Mathematics,
	TU Darmstadt, Schlossgartenstr. 7, 64289 Darmstadt, Germany}
\email{hieber@mathematik.tu-darmstadt.de}

\author[Hussein]{Amru Hussein} 
\address{Departement of Mathematics,
	TU Darmstadt, Schlossgartenstr. 7, 64289 Darmstadt, Germany}
\email{hussein@mathematik.tu-darmstadt.de}

\author[Kashiwabara]{Takahito Kashiwabara}
\address{Graduate School of Mathematical Sciences, The University of Tokyo, 3-8-1 Komaba, Meguro, Tokyo 153-8914, Japan}
\email{tkashiwa@ms.u-tokyo.ac.jp}

\begin{document}

\begin{abstract}
Consider the $3$-d primitive equations in a layer domain $\Omega=G \times (-h,0)$, $G=(0,1)^2$, subject to mixed Dirichlet and Neumann boundary conditions at $z=-h$ and $z=0$, respectively, and the 
periodic lateral boundary condition. It is shown that this equation is globally, strongly well-posed for arbitrary large data of the form 
$a=a_1 + a_2$, where $a_1\in C(\overline{G};L^p(-h,0))$, $a_2\in L^{\infty}(G;L^p(-h,0))$ for $p>3$,  and where $a_1$ is periodic in the horizontal variables and $a_2$ is sufficiently small. 
In particular, no differentiability condition on the data is assumed. The approach relies on $L^\infty_HL^p_z(\Omega)$-estimates for terms of the form 
$t^{1/2} \lVert \partial_z e^{t\Aipos}\mathbb{P}f \rVert_{L^\infty_H L^p_z(\Omega)}\le C e^{t\beta} \lVert f \rVert_{L^\infty_H L^p_z (\Omega)}$ for $t>0$, where $e^{t\Aipos}$ denotes the hydrostatic Stokes semigroup.  
The difficulty in proving estimates of this form  is that the hydrostatic Helmholtz projection $\mathbb{P}$ fails to be bounded with respect to the $L^\infty$-norm.
The global strong well-posedness result is then obtained by an iteration scheme, splitting the data into a smooth and 
a rough part and by combining a reference solution for smooth data with an evolution equation for the rough part. 
\end{abstract}

\maketitle

\section{Introduction}\label{sec:intro}

The primitive equations are a model for oceanic and atmospheric dynamics and are derived from the Navier-Stokes equations by assuming a hydrostatic balance for the pressure term, see 
\cite{Lionsetall1992, Lionsetall1992_b, Lionsetall1993}. These equations are known to be globally and  strongly well-posed in the three dimensional setting for arbitrarily large data belonging 
to $H^1$ by the celebrated  result of  Cao and Titi \cite{CaoTiti2007}. The latter considers the case of Neumann boundary conditions and this result also holds true for the case mixed 
Dirichlet and Neumann boundary conditions, again for data in $H^1$, as shown by Kukavika and Ziane \cite{Ziane2007}. 

Several approaches have been developed in the last years aiming for extending the above two results to the case of rough initial data. One approach is based on the theory of 
weak solutions, see e.g. \cite{LiTiti2015, TachimMedjo2010, Kukavicaetall2014, Ziane2009}. Although the existence of weak solutions to the primitive equations for initial data in $L^2$ 
is known since the pioneering work by Lions, Temam and Wang \cite{Lionsetall1992}, its uniqueness remains an open problem until today. Li and Titi \cite{LiTiti2015} proved  uniqueness of weak solutions 
assuming that the initial data are small $L^\infty$-perturbations of continuous data or data belonging to $\{v\in L^6 \colon \partial_z v\in L^2\}$, where $z$ denotes the vertical variable. 
By a weak-strong uniqueness argument, these unique weak solutions regularize and even become strong solutions. For a survey of known results, see also \cite{LiTiti2016}.

A different  approach to the primitive equations is based on a semilinear evolution equation for the hydrostatic Stokes operator within the  $L^p$-setting, see \cite{HieberKashiwabara2015}. 
There, the existence of a unique, global, strong solution to the primitive equations for initial data belonging to $H^{2/p,p}$ was proved for the case of mixed Dirichlet-Neumann boundary conditions. 
This approach was transfered   in \cite{NeumannNeumann, GigaGriesHieberHusseinKashiwabara2017} to the case of pure Neumann boundary conditions and global,  
strong well-posedness of the primitive equations was obtained for data $a$ of the form 
$a=a_1 + a_2$, where $a_1\in C(\overline{G};L^1(-h,0))$ and $a_2\in L^{\infty}(G;L^1(-h,0))$ with $a_2$ being small. These spaces are scaling invariant and represent the anisotropic character of the 
primitive equations.

Note that the choice of boundary conditions has a severe impact on the linearized primitive equations. In the setting of layer domains, i.e., 
$\Omega= G\times (-h,0)\subset \R^3$ with $G= (0,1)^2$ and $h>0$, this is illustrated best by the hydrostatic Stokes operator $\Aipos$. The latter can be represented formally by the differential expression 
\begin{align}\label{eq:Ap}
	\Ae v = \Delta v + \frac{1}{h}\nabla_H (-\Delta_H)^{-1}\text{div}_H \Big(\restr{\partial_z v}{z=-h}\Big),
\end{align}
restricted to hydrostatically solenoidal vector fields, where for $z=-h$ Dirichlet and for $z=0$ Neumann boundary conditions are imposed and periodicity is assumed horizontally, see \cite{GGHHK17} for details. In particular, 
in the case of pure Neumann boundary conditions, the hydrostatic Stokes operator reduces to the Laplacian, i.e. $\Aipos v =\Delta v$.

It is the aim of this article to study properties of the hydrostatic Stokes semigroup and terms of the form $\nabla e^{t\Aipos} \mathbb{P}$ on spaces of bounded functions. 
These properties yield then the global, strong well-posedness result of the primitive equations in the case of mixed Dirichlet-Neumann boundary conditions.  More precisely, we prove global, strong well-posedness of the primitive 
equations for initial data  of the form 
$$
a=a_1 + a_2, \quad a_1\in C(\overline{G};L^p(-h,0)), \quad \mbox{and} \quad a_2\in L^{\infty}(G;L^p(-h,0)) \quad \mbox{for} \quad p>3,
$$
where $a_1$ is periodic in the horizontal variables and $a_2$ is sufficiently small. Our strategy is to introduce  a reference solution for the smoothened part of the initial data and to combine this 
with an evolution equation approach for the remaining rough part. 

The main difficulty when dealing with the primitive equations on spaces of bounded functions is that the hydrostatic Helmholtz projection $\mathbb{P}$ {\em fails to be bounded} with respect to 
the $L^\infty$-norm. This is similar to the case of the classical Stokes semigroup, for which $L^\infty$-theory was developed in \cite{AG13} and \cite{AGH15}. 

In Sections 6 and 7 we prove that the combination of the three main players, $\nabla$, $\mathbb{P}$, $e^{t\Aipos}$, nevertheless give rise to bounded operators on  
$L^\infty_HL^p_z(\Omega)$, which in addition satisfy typical global, second order parabolic decay estimates of the form
\begin{align*}
t^{1/2} \lVert \partial_i e^{t\Aipos}\mathbb{P}f \rVert_{L^\infty_H L^p_z(\Omega)} &\le C e^{t\beta} \lVert f \rVert_{L^\infty_H L^p_z (\Omega)}, \\
t^{1/2} \lVert e^{t\Aipos}\mathbb{P}\partial_j f  \rVert_{L^\infty_H L^p_z(\Omega)}& \le C e^{t\beta}\lVert f\rVert_{L^\infty_H L^p_z (\Omega)}, \\
t\lVert \partial_i e^{t\Aipos}\PP \partial_j f\rVert_{L^\infty_H L^p_z(\Omega)} &\le C e^{\beta t}\lVert f\rVert_{L^\infty_H L^p_z(\Omega)},
\end{align*}		
for $t>0$, where $\partial_i, \partial_j\in \{\partial_x,\partial_y,\partial_z\}$.

Note that the choice of the boundary conditions involved affects to a very great extent the difficulty in  proving these estimates. 
For the case of mixed Dirichlet-Neumann boundary conditions, our approach relies on the representation \eqref{eq:Ap} of the linearized problem. 
The constraint $p>3$ arises from embedding properties for the reference solution and estimates for the linearized problem in $L^{\infty}(G;L^p(-h,0))$. 

Our approach is based on an iteration scheme, which is  inspired by the classical schemes  to the Navier-Stokes equations. 
Here, the iterative construction of a unique, local solution relies on $L^\infty_HL^p_z(\Omega)$-estimates for the crucial terms of the 
form $e^{t\Aipos}\mathbb{P}\div (u \otimes v)$, where $u=(v,w)$ is the full velocity and $v$ its horizontal component.  
Let us note that the above linear estimates are of independent interest  for further considerations.

The use of a reference solution allows us to obtain the smallness condition on the $L^\infty_H L^p_z$-perturbation $a_2$ of $a_1$ by means of an absolute constant, while for Neumann boundary conditions 
it is needed that $a_2$ is small compared to $a_1$, cf.\@ \cite{NeumannNeumann}. Also, Li and Titi assume in \cite{LiTiti2015} that $a_2$ is small compared to the $L^4$-norm of $a_1$.

Comparing our result with the one by  Li and Titi in \cite{LiTiti2015}, which has been obtained for Neumann boundary conditions, we observe that the initial data allowed in our approach are of anisotropic 
nature and require no conditions on the derivatives of the initial data, such as e.g. $\partial_z v \in L^2$ as in \cite{LiTiti2015}.

This article is structured as follows: In Section~\ref{sec:pre} we collect preliminary facts and fix the notation. In Section~\ref{sec:main} we state our main results
concerning the global strong well-posedness of the primitive equations for rough data and the crucial estimates for the linearized problem. The proof of our main results starts with  a discussion of 
anisotropic $L^p$-spaces in Section~\ref{sec:spaces}, which is followed in Section~\ref{sec:laplace} by estimates for the Laplacian  in anisotropic spaces.  
The subsequent  Sections ~\ref{sec:stokesEasy} and \ref{sec:stokesHard} are devoted to the development of an $L^{\infty}(G;L^p(-h,0))$-theory for the hydrostatic Stokes equations and its 
associated resolvent problem. Finally, in Section~\ref{sec:proofs} we present our iteration scheme yielding the global, strong well-posedness of the primitive equations for rough initial data.

\section{Preliminaries}\label{sec:pre}

Let $\Omega=G\times(-h,0)$ where $G=(0,1)^2$. We consider the primitive equations on $\Omega$ given by
		\begin{align}\label{eq:PrimitiveEquations}
		\begin{array}{rll}
		\partial_t v-\Delta v+(u\cdot \nabla)v+\nabla_H \pi
		&=0 & \text{ on } \Omega\times(0,\infty),
		\\
		\partial_z \pi&=0 & \text{ on } \Omega\times (0,\infty),
		\\
		\text{div}_H \overline{v}&=0 & \text { on } G\times(0,\infty),
		\\
		v(0)&=a & \text { on } \Omega,
		\end{array}
		\end{align}
using the notations $\div_H v = \partial_x v_1+\partial_y v_2$ and $\nabla_H \pi = (\partial_x \pi, \partial_y \pi)^T$, while $\overline{v}=\frac{1}{h}\int_{-h}^0 v(\cdot,z)\,dz$ is the vertical average, $\pi\colon G\to \R$ denotes the surface pressure, $u=(v,w)$ is the velocity field with horizontal and vertical components $v:\Omega\to \R^2$ and $w:\Omega\to \R$ respectively, where $w=w(v)$ is given by the relation
		\begin{align}\label{eq:WRelation}
		w(x,y,z)=-\int_h^z \text{div}_H v(x,y,r)\,dr.
		\end{align}
This is supplemented by mixed Dirichlet and Neumann boundary conditions
		\begin{align}\label{eq:bc}
		\partial_z v =0  \text{ on } \Gamma_u\times(0,\infty), \quad
		\pi, v  \text { periodic }  \text{ on } \Gamma_l\times(0,\infty),  \quad
		v =0  \text{ on } \Gamma_b\times(0,\infty),
		\end{align}
where the boundary is divided into $\Gamma_u=G\times \{0\}$, $\Gamma_l=\partial G\times [-h,0]$ and $\Gamma_b=G\times\{0\}$. 

In the following we will be dealing with anisotropic $L^p$-spaces on cylindrical sets of the type $U= \Omega$ or $U=\R^2 \times \R$. More precisely, if $U=U'\times U_3\subset \R^2\times \R$ is a product of measurable sets and $q,p\in[1,\infty]$ we define
		\begin{align*}
		L^q_H L^p_z(U)
		:=L^q(U';L^p(U_3))
		:=\{f:U\to \K \text{ measurable}, \lVert f\rVert_{L^q_H L^p_z(U)}<\infty\},
		\end{align*}
for $\K\in\{\R,\C\}$ with norm
		\begin{align*}
		\lVert f\rVert_{L^q_H L^p_z(U)}:=\begin{cases}
		& \left(\int_{U'} \lVert f(x',\cdot)\rVert^q_{L^p(U_3)}\,dx'\right)^{1/q},
		\quad q\in [1,\infty),
		\\
		&
		\text{ess sup}_{x'\in U'} \lVert f(x',\cdot)\rVert_{L^p(U_3)}, \quad q=\infty.
		\end{cases}
		\end{align*}
Endowed with this norm, $L^q_H L^p_z(U)$ is a Banach space for all $p,q\in [1,\infty]$.

We will denote the $W^{k,p}$-closure of $C_{\text{per}}^\infty(\overline{\Omega})$ by $W^{k,p}_{\text{per}}(\Omega)$, 
where $C_{\text{per}}^\infty(\overline{\Omega})$ denotes the space of smooth functions $v$ on $\overline{\Omega}$ that such that $\partial^\alpha_x v$ and $\partial^\alpha_y v$ are periodic on $\Gamma_l$ with 
period $1$ in the variables $x$ and $y$ for all $\alpha \in \N$, but not necessarily periodic with respect to the vertical direction $z$. Moreover, by $C^{m,\alpha}(\overline{\Omega})$, $C^{m,\alpha}(\overline{G})$ we denote the spaces of 
$m$-times differentiable functions with H\"older-continuous derivatives of exponents $\alpha\in (0,1)$ and the subspaces of functions periodic on $\Gamma_l$ and $\partial G$ will be denoted 
by $C^{m,\alpha}_{\text{per}}(\overline{\Omega})$ and $C^{m,\alpha}_{\text{per}}(\overline{G})$, respectively. For a Banach space $E$ we denote by $C_{\text{per}}([0,1]^2;E)$ the set of 
continuous functions $f:[0,1]^2\to E$ such that $f(0,y)=f(1,y)$ and $f(y,0)=f(y,1)$ for all $x,y\in [0,1]$.

In order to include the condition $\div_H \overline{v}=0$ one defines the \textit{hydrostatic Helmholtz projection} $\PP$ as in \cite{HieberKashiwabara2015, GGHHK17} using the two-dimensional Helmholtz projection $Q$ with periodic boundary conditions given by $Qg=g-\nabla_H \pi$ for $g:G\to \R^2$ solving $\Delta_H \pi = \text{div}_H g$ for $\pi$ periodic on $\partial G$,
where $\Delta_H g=\partial_x^2 g+\partial_y^2 g$. The hydrostatic Helmholtz projection is then defined as 
\begin{align*}
\PP f=f-(1-Q)\overline{f}= f + \frac{1}{h}\nabla_H (-\Delta_H)^{-1}\text{div}_H\overline{f}  = f-\nabla_H \pi.
\end{align*}
		%
 The range of $\PP\colon L^p(\Omega)^2\rightarrow L^p(\Omega)^2$, $p\in (1,\infty)$, is denoted by $\lpso$ and is given by 
		\[
		\overline{\{v\in C_{\text{per}}^\infty(\overline{\Omega})^2: \text{div}_H \overline{v}=0\}}^{\norm{\cdot}_{L^p(\Omega)}}.
		\]
Further characterizations of $\lpso$ are given in \cite[Proposition 4.3]{HieberKashiwabara2015}.


Since $\mathbb{P}$ fails to be bounded on $L^\infty(\Omega)^2$ it is not evident which space  
is a suitable substitute for $\lpso$ in the case $p=\infty$. In this article, we will be considering the spaces
		\begin{align}\label{eq:Xpspaces}
		\Xip:=C_{\text{per}}([0,1]^2;L^p(-h,0))^2
		\quad \hbox{and} \quad
		\Xipos:=\Xip\cap \lpso, \quad p\in (1,\infty).
		\end{align}
		
The linearization of equation \eqref{eq:PrimitiveEquations}, called the \textit{hydrostatic Stokes equation}, is  given by 
		\begin{align}\label{eq:hydrostaticStokes}
		\partial_t v - \Delta v + \nabla_H \pi =f, \quad \div_H \overline{v}=0, \quad v(0)=a
		\end{align}
and subject to boundary conditions \eqref{eq:bc}. The dynamics of this evolution equation is governed by the hydrostatic Stokes operator, and its $\Xipos$-realization $\Aipos$ 
is given by
	\begin{align*}
		\Aipos v := \Ae v, \quad D(A^{\os}_\infty)
		=\{
		v\in W^{2,p}_{\text{per}}(\Omega)^2\cap \Xipos: 
		\restr{\partial_z v}{\Gamma_u}=0, 
		\restr{v}{\Gamma_b}=0, 
		\Ae v\in \Xipos
		\},
	\end{align*}
where $\Ae v$ is defined by \eqref{eq:Ap}. It wil be proved that $\Aipos$ generates a strongly continuous, analytic semigroup $e^{t\Aipos}$ on $\Xipos$. 
Information on the linear theory in $\lpso$ for $p\in (1,\infty)$ can be found in \cite{GGHHK17}. 

\section{Main results}\label{sec:main}

Our first main result concerns the global well-posedness of the primitive equations for \textit{arbitrarily large} initial data in $\Xipos$, while the second result extends this situation to the 
case of small perturbations in $L^\infty_H L^p_z(\Omega)$. Here, a \textit{strong solution} means -- as in \cite{HieberKashiwabara2015} -- a solution $v$ to the primitive equations satisfying
		\begin{align}\label{eq:StrongSolution}
		v\in C^1((0,\infty);L^p(\Omega))^2\cap C((0,\infty);W^{2,p}(\Omega))^2.
		\end{align}
Our third main result concerns $L^\infty_HL^p_Z$-estimates for the hydrostatic Stokes semigroup.  These estimates are essential for proving the above two results on the non-linear problem. They are also of 
independent interest.

\begin{theorem}\label{MainTheorem}
Let $p\in (3,\infty)$. Then for all $a\in \Xipos$ there exists a unique, global, strong solution $v$ to the primitive equations \eqref{eq:PrimitiveEquations} with $v(0)=a$ satisfying 
		\[
		v\in C([0,\infty);\Xipos),
		\quad
		t^{1/2}\nabla v\in L^\infty((0,\infty);\Xip), 
		\quad
		\limsup_{t\to 0+}t^{1/2}\lVert \nabla v(t) \rVert_{L^\infty_H L^p_z(\Omega)}=0.
		\]
The corresponding pressure satisfies
		\[
		\pi \in C((0,\infty);C^{1,\alpha}([0,1]^2)),
		\quad \alpha\in (0,1-3/p)
		\]
and is unique up to an additive constant.
\end{theorem}
%

%

%

\begin{theorem}\label{PerturbationTheorem}
Let $p\in (3,\infty)$. Then there exists a constant $C_0>0$ such that if $a=a_1+a_2$ with $a_1\in \Xipos$ and $a_2\in L^\infty_H L^p_z(\Omega)^2\cap \lpso$ with
		\[
		\lVert a_2\rVert_{L^\infty_H L^p_z(\Omega)}\le C_0,
		\]
then there exists a unique, global, strong solution $v$ to the primitive equations \eqref{eq:PrimitiveEquations} with $v(0)=a$ satisfying
		\[
		v\in C([0,\infty);\lpso)
		\cap L^{\infty}((0,T); L^\infty_H L^p_z(\Omega))^2
		\]
as well as 
		\[
		t^{1/2}\nabla v\in L^\infty((0,\infty);\Xip),
		\quad
		\limsup_{t\to 0+}t^{1/2}\lVert \nabla v \rVert_{L^\infty_H L^p_z(\Omega)}
		\le C \lVert a_2\rVert_{L^\infty_H L^p_z},
		\]
where $C>0$ does not depend on the data, and the pressure has the same regularity as in Theorem~\ref{MainTheorem}.	
\end{theorem}

Taking advantage of the regularization of solutions for $t>0$ one passes into the setting discussed in  \cite{HieberKashiwabara2015} and \cite{GigaGriesHieberHusseinKashiwabara2017}, and thus we 
obtain the following corollary. 

\begin{corollary}\label{cor:smoothness}
For $t>0$ the solution $v, \pi$ in Theorem~\ref{MainTheorem} and Theorem~\ref{PerturbationTheorem} are real analytic in time and space, and the velocity $v$ decays exponentially as $t \to \infty$.
\end{corollary}

Our main result on the hydrostatic semigroup acting on $\Xipos$ reads as follows.  

\begin{theorem}\label{thm:semigroup}
Let $p\in (3,\infty)$. Then the following assertions hold true: \\
a) $\Aipos$ is the generator of a strongly continuous, analytic and exponentially stable semigroup $e^{t\Aipos}$ on $\Xipos$  of angle $\pi/2$. \\
b) There exist constants $C>0$, $\beta>0$ such that for $\partial_i, \partial_j\in \{\partial_x,\partial_y,\partial_z\}$ 
\begin{align*}
\tag{i}	 t^{1/2}\lVert \partial_j e^{t\Aipos} f\rVert_{L^\infty_H L^p_z(\Omega)}
		&\le C e^{\beta t}\lVert f\rVert_{L^\infty_H L^p_z(\Omega)}, \quad t>0, \, f \in \Xipos,
		\\
	\tag{ii}	 t^{1/2}\lVert e^{t\Aipos}\PP \partial_j f\rVert_{L^\infty_H L^p_z(\Omega)}
		&\le C e^{\beta t}\lVert f\rVert_{L^\infty_H L^p_z(\Omega)}, \quad t>0, \, f \in \Xipos,\\
	\tag{iii}	 t\lVert \partial_i e^{t\Aipos}\PP \partial_j f\rVert_{L^\infty_H L^p_z(\Omega)}
		&\le C e^{\beta t}\lVert f\rVert_{L^\infty_H L^p_z(\Omega)}, \quad t>0, \, f \in \Xipos;
		\end{align*}
\item[c)] 
For all $f\in \Xipos$ 
\begin{align*} 
		\lim_{t \to 0+} t^{1/2}\lVert \nabla e^{t\Aipos}f\rVert_{L^\infty_H L^p_z(\Omega)}=0.
		\end{align*}
\end{theorem}

\begin{remarks}\label{rem:main}
a) We note that when in the situation of Theorem~\ref{PerturbationTheorem} the initial data do not belong to $\Xip$, i.e. when $a_2\neq 0$, the solution fails to be continuous at $t=0$ with respect 
to the $L^\infty_H L^p_z$-norm. \\
b) The condition $p>3$ is due to the embeddings 		
		\[
		v_{\text{ref}}(t_0)
		\in B^{2-2/q}_{pq}(\Omega)^2
		\hookrightarrow C^1(\overline{\Omega})^2 \quad \hbox{and} \quad  W^{2,p}(\Omega) \hookrightarrow C^{1,\alpha}(\overline{\Omega})
		\quad \hbox{for } p\in (3,\infty),
		\]
cf. \cite[Section 3.3.1]{Triebel}. Since the two-dimensional Helmholtz projection $Q$ fails to be bounded with respect to the $L^\infty$-norm, we instead estimate it in spaces of H{\"o}lder continuous functions $C^{0,\alpha}_{\text{per}}([0,1]^2)=C^{0,\alpha}(\mathbb{T}^2)$ for $\alpha \in (0,1)$ where $\mathbb{T}^2$ denotes the two-dimensional torus.  
In fact $Q$ is bounded with respect to the $C^{0,\alpha}$-norm. This follows by the theory of Fourier multipliers on Besov spaces, compare e.g. \cite[Theorem 6.2]{Amann1997} for the whole space, and the periodic case follows using periodic extension. 
%
\\
c) In Theorem~\ref{thm:semigroup} one can even consider $f\in L^{\infty}_HL^p_z(\Omega)^2$ for $p\in (3,\infty)$. Then the corresponding semigroup is still analytic, but it fails to be strongly continuous. The estimates $(i)-(iii)$ still hold, whereas property $c)$ in Theorem~\ref{thm:semigroup} has to be replaced by   
			\begin{align*} 
			\limsup_{t \to 0+} t^{1/2}\lVert \nabla e^{t\Aipos}v\rVert_{L^\infty_H L^p_z(\Omega)} \leq C \lVert v\rVert_{L^\infty_H L^p_z(\Omega)}
			\end{align*}
			for some $C>0$, where with a slight abuse of notation $e^{t\Aipos}$ denotes also the hydrostatic  Stokes semigroup on $L^\infty_H L^p_z(\Omega)$.  
\\			
d ) Some words about our strategy for proving the global well-posedness results are in order:
\begin{itemize}
\item[(i)] We will first construct a local, mild solution to the problem \eqref{eq:PrimitiveEquations}, i.e. a function 
satisfying the relation
		\begin{align}\label{eq:mildsolution}
		v(t)=e^{t\Aipos}a+\int_0^t e^{(t-s)\Aipos}\PP F(v(s))\,ds,
		\quad t\in (0,T)
		\end{align}
for some $T>0$, where $F(v)=-(u\cdot\nabla) v$. We will then show that $v$ regularizes for $t_0>0$ and using the result of \cite[Theorem 6.1]{HieberKashiwabara2015} or 
\cite[Theorem 3.1]{GigaGriesHieberHusseinKashiwabara2017}, we may 
take $v(t_0)$ as a new initial value to extend the mild solution to a global, strong solution 
on $(t_0,\infty)$ and then on $(0,\infty)$ by uniqueness. The additional regularity for $t\to 0+$ results form the construction of the mild solutions. 
\item[(ii)] In order to construct a mild solution 
we decompose $a=a_{\text{ref}}+a_0$ such 
that $a_{\text{ref}}$ is sufficiently smooth and $a_0$ can be taken to be arbitrarily small.
\item[(iii)]Using previously established results concerning the existence of solutions to the primitive equations for {\em smooth} data, we obtain a reference solution $v_{\text{ref}}$ and 
construct then $V:=v-v_{\text{ref}}$ via an iteration scheme using $L^\infty$-type estimates for terms of the form $\nabla e^{t\Aipos}\mathbb{P}$ given in Theorem~\ref{thm:semigroup}.
\end{itemize}
\end{remarks}

\section{Properties of anisotropic spaces}\label{sec:spaces}

In this section, we will discuss properties of anisotropic $L^q$-$L^p$-spaces. We will write $C(U';L^p(U_3))$ for the set of continuous $L^p(U_3)$-valued functions on $U'$ and likewise
		\[
		L^q(U';C(U_3))
		:=\{ f\in L^q_H L^\infty_z(U): 
		f(x',\cdot)\in C(U_3) \text{ for almost all } x'\in U'
		\},
		\]
and $C_c(U';L^p(U_3))$ and $L^q(U';C_c(U_3))$ for the subsets of functions with compact support in horizontal and vertical variables, respectively. For $p,q\in[1,\infty)$ the space  
$C^\infty_c(U)$ is dense in these spaces as well as in $L^q_H L^p_z(U)$, and furthermore we have
		\begin{align*}
		\overline{C^\infty_c(\R^3)}^{\norm{\cdot}_{L^\infty_H L^p_z}} 
		= C_0(\R^2;L^p(\R)),  
		\quad 
		\overline{C^\infty_c(\R^3)}^{\norm{\cdot}_{L^q_H L^\infty_z}} 
		= L^q(\R^2;C_0(\R)), 
		\end{align*}
as well as
		\begin{align*}
		\overline{C_{\text{per}}^\infty(\overline{\Omega})^2}
		^{\norm{\cdot}_{L^\infty_H 	L^p_z}} 
		= \Xip,
		\quad 
		\overline{C_{\text{per}}^\infty(\overline{\Omega})}
		^{\norm{\cdot}_{L^q_H L^\infty_z}} 
		= L^q(G;C[-h,0]).
		\end{align*} 
Observe thst even $C^\infty_{\text{per}}([0,1]^2;C^\infty_c(-h,0))^2$ is dense in $\Xip$ and $L^q_H L^p_z(\Omega)^2$. If $p=q=\infty$, then  
		\[
		\overline{C^\infty_c(\R^3)}^{\norm{\cdot}_{L^\infty_H L^\infty_z}} 
		= C_0(\R^3),
		\quad \overline{C_{\text{per}}^\infty(\overline{\Omega})}^{\norm{\cdot}_{L^\infty_H L^\infty_z}} 
		= C_\text{per}([0,1]^2;C[-h,0]).
		\]
Here $C_0(\R^d)$ denotes the set of functions vanishing at infinity.
These density results follow from the fact that if $E$ is a Banach space over $\K\in \{\R,\C\}$, then the linear space generated by elementary tensor functions $f\otimes e$ for 
measurable $f:U'\to \K$ and $e\in E$ is dense in $L^q(U';E)$ for $q\in [1,\infty)$, since it contains the simple $E$-valued functions. It is also dense in $C_0(U';E)$, if one only 
considers continuous functions $f$, due to a generalization of the Stone-Weierstrass theorem, see e.g. \cite{Khan}. 

In the case that $U\subset \R^3$ is bounded, we also have
		\[
		L^{q_1}_H L^{p_1}_z(U)\hookrightarrow L^{q_2}_H L^{p_2}_z(U)
		\]
whenever $q_1\ge q_2$ and $p_1\ge p_2$. See \cite[Section 5]{HieberKashiwabara2015} for more details.

Another important property of the $L^q_H L^p_z$-norm is its behaviour  under operations like multiplication and convolution. For the former one, we  obviously obtain 
		\[
		\lVert fg\rVert_{L^q_H L^p_z(U)}
		\le \lVert f\rVert_{L^{q_1}_H L^{p_1}_z(U)} \lVert g\rVert_{L^{q_2}_H L^{p_2}_z(U)}
		\]
whenever $1/p_1+1/p_2=1/p$ and $1/q_1+1/q_2=1/q$. For the latter one, the following variant of Young's inequality holds true.

\begin{lemma}{\cite[Theorem 3.1]{GreySinnamon}}\label{YoungAnisotropic}. 
Let $f\in L^q_H L^p_z (\R^3)$ for $p,q\in [1,\infty]$ and  $g\in L^1(\R^3)$. Then 
		\begin{align*} 
		\lVert g\ast f\rVert_{L^q_H L^p_z (\R^3)}
		\le \lVert g\rVert_{L^1(\R^3)}
		\lVert f\rVert_{L^q_H L^p_z (\R^3)}.
		\end{align*}
\end{lemma}

\section{Linear estimates for the Laplace operator}\label{sec:laplace}
In this section we establish resolvent and semigroup estimates for Laplace operators with a focus on anisotropic $L^q_HL^p_z$-spaces, where $p,q\in [1,\infty]$. 
%

First, we consider the resolvent problem for the 
Laplacian on the full space for 
$$\lambda\in \Sigma_\theta=\{
		\lambda\in\C\setminus\{0\}:\lvert\text{arg}(\lambda)\rvert<\theta\}, \quad \theta\in (0,\pi),$$ 
		i.e.
		\begin{align}\label{eq:LaplaceResolventFullspace}
		\Delta v - \lambda v=f  \text{ on } \R^3, \quad f\in C^\infty_c(\R^3),
		\end{align}
and for $\partial_j\in\{\partial_x,\partial_y,\partial_z\}$
		\begin{align}\label{eq:LaplaceResolventDerivativeFullspace}
		\Delta w - \lambda w=\partial_j f  \text{ on } \R^3, \quad f\in C^\infty_c(\R^3).
		\end{align}
		%
It is well known that the solution to problem \eqref{eq:LaplaceResolventFullspace} is given by the convolution $v=K_\lambda \ast f$ and the one to problem \eqref{eq:LaplaceResolventDerivativeFullspace} by 
$v=\partial_j K_\lambda\ast f$, where $K_\lambda$ is explicitly given by
		\begin{align*}
		K_\lambda(x)
		=\frac{1}{4\pi}
		\frac{
		e^{
		-\lambda^{1/2}\lvert x\rvert
		}
		}{\lvert x\rvert}, 
		\quad x\in\R^3\setminus\{0\}.
		\end{align*}
Using this representation one easily obtains the following uniform $L^1(\R^3)$-estimates.

\begin{lemma}
For all $\theta\in(0,\pi)$ there exists $C_\theta>0$ such that for all $\lambda \in \Sigma_{\theta}$ one has
		\[
		\lvert \lambda \rvert \cdot  \lVert K_\lambda\rVert_{L^1(\R^3)}
		+\lvert \lambda \rvert^{1/2} \lVert \nabla K_\lambda\rVert_{L^1(\R^3)}
		\le C_\theta.
		\]
\end{lemma}
\begin{proof}
Set $\psi:=\text{arg}(\lambda)\in (-\theta,\theta)$. Since $K_\lambda$ is radially symmetric we use spherical coordinates to obtain
		\[
		\int_{\R^3} \lvert K_\lambda(x)\rvert\,dx
		=\int_0^\infty r e^{-\lvert \lambda\rvert^{1/2}\cos(\psi/2)r}
		\,dr
		\]
as well as
		\[
		\int_{\R^3} \lvert \nabla K_\lambda(x)\rvert\,dx
		\le 
		\int_0^\infty 
		(1+\lvert \lambda\rvert^{1/2}r)e^{-\lvert \lambda\rvert^{1/2}\cos(\psi/2)r}
		\,dr.
		\]
So, $\lvert \lambda \rvert \cdot \lVert K_\lambda\rVert_{L^1(\R^3)}	=\text{sec}(\psi/	2)^2$ and  
$\lvert \lambda\rvert^{1/2} \lVert \nabla K_\lambda \rVert_{L^1(\R^3)}
		\le \sec(\psi/2)+\sec(\psi/2)^2$, and thus
we obtain the desired result. 				
		%
		%
\end{proof}

From this and Young's inequality for convolutions in anisotropic spaces, cf. Lemma~\ref{YoungAnisotropic}, one immediately obtains suitable $L^q_H L^p_z$-estimates for the resolvent problems \eqref{eq:LaplaceResolventFullspace} and \eqref{eq:LaplaceResolventDerivativeFullspace} for $q,p\in [1,\infty]$. 

\begin{corollary}\label{LemmaLaplaceResolventFullspace}
Let $\lambda \in \Sigma_{\theta}$ for some $\theta\in(0,\pi)$. Assume one of the following cases:
\begin{itemize}
\item[(i)] $p,q\in [1,\infty)$ and $f\in L^q_H L^p_z (\R^3)$, or
\item[(ii)] $p\in [1,\infty)$, $q=\infty$, and $f\in L^q_H L^p_z (\R^3)$ with compact support in horizontal direction, or
\item[(iii)] $p=\infty$, $q\in [1,\infty)$, and $f\in L^q_H L^p_z (\R^3)$ with compact support in vertical direction.
\end{itemize} 
Then the functions
\begin{align*}
 v=K_\lambda\ast f \quad \hbox{ and } \quad w=\partial_j K_\lambda\ast f
\end{align*}
are the unique solutions to the problems \eqref{eq:LaplaceResolventFullspace} and \eqref{eq:LaplaceResolventDerivativeFullspace} in $L^q_HL^p_z(\R^3)$, respectively, and
%
there exists a constant $C_\theta>0$ such that 
		%
		\begin{align}\label{eq:AnisotropicResolventFullspace}
		\lvert \lambda\rvert \cdot \lVert v\rVert_{L^q_H L^p_z (\R^3)}
		+\lvert \lambda\rvert^{1/2} \lVert \nabla v\rVert_{L^q_H L^p_z (\R^3)}
		+\lVert \Delta v\rVert_{L^q_H L^p_z (\R^3)}
		\le C_\theta \lVert f\rVert_{L^q_H L^p_z (\R^3)},\\
		%
		\lvert \lambda\rvert^{1/2} \lVert w\rVert_{L^q_H L^p_z (\R^3)}
		\le C_\theta \lVert f\rVert_{L^q_H L^p_z (\R^3)}.\label{eq:AnisotropicResolventDerivativeFullspace}
		\end{align}
\end{corollary}

\begin{remark}
In the case $q,p\in[1,\infty)$ we have that $C^\infty_c(\R^3)$ is dense in $L^q_H L^p_z(\R^3)$, so we may assume that $f$ is essentially bounded and has compact support, yielding $\partial_i (K_\lambda \ast f)=(\partial_i K_\lambda)\ast f$. In the cases where $q$ and/or $p$ is infinite we add this  as an assumption. 
\end{remark}

We now investigate for the Laplacian on $\Omega$ with boundary conditions \eqref{eq:bc} the resolvent problems 
		\begin{align}\label{eq:LaplaceResolventOmega}
		\lambda v-\Delta v=f  \text{ on } \Omega,
		\end{align}
and for $\partial_i \in\{\partial_x,\partial_y,\partial_z\}$
		\begin{align}\label{eq:LaplaceResolventDerivativeOmega}
		\lambda w-\Delta w=\partial_i f \text{ on } \Omega.
		\end{align}

\begin{lemma}\label{LemmaAnisotropicLaplaceResolventOmega}
Let  $\theta\in (0,\pi)$ and
$f\in L^q_H L^p_z(\Omega)$ for $q\in [1,\infty], p\in [1,\infty)$.
%
%
Then there exists $\lambda_0>0$ such that for $\lambda \in \Sigma_{\theta}$  with $\lvert \lambda\rvert\ge \lambda_0$ the problems \eqref{eq:LaplaceResolventOmega} and \eqref{eq:LaplaceResolventDerivativeOmega} have unique solutions $v\in L^q_H L^p_z(\Omega)$ and $w\in L^q_H L^p_z(\Omega)$, respectively, and there exists a constant $C_{\theta}>0$
such that		%
		\begin{align}\label{eq:AnisotropicResolventOmega}
		\lvert \lambda \rvert \cdot \lVert v\rVert_{L^q_H L^p_z(\Omega)}
		+ \lvert \lambda \rvert^{1/2} \lVert \nabla v\rVert_{L^q_H L^p_z(\Omega)}
		+ \lVert \Delta v\rVert_{L^q_H L^p_z(\Omega)}
		\le C_{\theta} \lVert f\rVert_{L^q_H L^p_z(\Omega)}, \\
		\label{eq:AnisotropicResolventDerivativeOmega}
		\lvert \lambda \rvert^{1/2}  \lVert w\rVert_{L^q_H L^p_z(\Omega)}
		\le C_{\theta} \lVert f\rVert_{L^q_H L^p_z(\Omega)}.
		\end{align}
In particular for $q=\infty$ and $p\in (2,\infty)$ one can chose $\lambda_0=0$. 
\end{lemma}
%

To prove this lemma, we will need some facts concerning isotropic $L^p$-spaces.
So, for $p\in (1,\infty)$ denote by $\Delta_p$ the Laplace operator on $L^p(\Omega)$ defined by
\begin{align*}
\Delta_p v=\Delta v, \quad D(\Delta_p)=\{
		v\in W^{2,p}_{\text{per}}(\Omega): 
		\restr{\partial_z v}{\Gamma_u}=0, 
		\restr{v}{\Gamma_b}=0
		\}.
\end{align*}		
%
		%
%
One has $\rho(-\Delta_p)\subset \C\setminus  [\delta,\infty)$, for some $\delta>0$, i.e. $0\in \rho(-\Delta_p)$, cf. \cite[Remark 8.23]{Nau2012}, and the resolvent satisfies for some $C_{\theta,p}>0$ the estimate
		\begin{align}\label{eq:LpResolventEstimateLaplace}
		\lvert \lambda \rvert \cdot \lVert (\Delta_p - \lambda)^{-1}f\rVert_{L^p(\Omega)}
		+ \Vert \Delta_p(\Delta_p - \lambda)^{-1}f\rVert_{L^p(\Omega)}
		\le C_{\theta,p}\lVert f\rVert_{L^p(\Omega)}, \quad f\in L^p(\Omega),
		\end{align}
where $\lambda\in \Sigma_\theta$, $\theta\in(0,\pi)$.
	 	%
 		%
%
%
Furthermore, $-\Delta_p$ possesses a bounded $\mathcal{H}^\infty$-calculus of angle $0$, see e.g. \cite{Nau2012}, 
and therefore
		\begin{align}\label{eq:Domains}
		D((-\Delta_p)^\vartheta)=[L^p(\Omega),D(\Delta_p)]_\vartheta\subset W^{2\vartheta,p}(\Omega), \quad \vartheta\in [0,1], 
		\end{align}
where $[\cdot,\cdot]$ denotes the complex interpolation functor. In particular $\partial_j(-\Delta_p)^{-1/2}$ is bounded on $L^p(\Omega)$ for $\partial_j\in\{\partial_x,\partial_y,\partial_z\}$ and by taking adjoints the same holds true for the closure of $(-\Delta_p)^{-1/2}\partial_j$.
%
%
%
%
%
%
%
This yields the estimates
		\begin{align}\label{eq:LpEstimateDerivativesLaplace}
		\begin{split}
		\lvert \lambda \rvert^{1/2}\lVert \partial_j(\Delta_p-\lambda)^{-1} f\rVert_{L^p(\Omega)}
		+\lvert \lambda \rvert^{1/2}\lVert 
		(\Delta_p-\lambda)^{-1}\partial_j f\rVert_{L^p(\Omega)}
		& \le C_{\theta,p} \lVert f\rVert_{L^p(\Omega)}, \quad f\in L^p(\Omega), \\
		\lVert \partial_j (\Delta_p-\lambda)^{-1}\partial_i f\rVert_{L^p(\Omega)}
		& \le C_{\theta,p}\lVert f\rVert_{L^p(\Omega)}, \quad f\in L^p(\Omega)
		\end{split}
		\end{align}
		for $\lambda\in \Sigma_\theta$, $\theta\in(0,\pi)$, and some $C_{\theta,p}>0$.
%
%

\begin{proof}[Proof of Lemma~\ref{LemmaAnisotropicLaplaceResolventOmega}]
First, we apply the following density arguments: 
\begin{itemize}
\item[(i)] For $q,p\in [1,\infty)$ and $f\in L^q_H L^p_z(\Omega)$ we assume that $f\in C^\infty_{\text{per}}([0,1]^2;C^\infty_c(-h,0))$ since \\ $C^\infty_{\text{per}}([0,1]^2;C^\infty_c(-h,0))$ is a dense subspace of $L^q_H L^p_z(\Omega)$. 
\item[(ii)] For $q=\infty$ and $f\in L^\infty(G;L^p(-h,0))$, we assume that $f\in L^\infty(G;C^\infty_c(-h,0))$, as the latter space is dense in $L^\infty_H L^p_z(\Omega)$.
\end{itemize}
%
%
%
%
In particular, in either case we may assume that $f=0$ on $\Gamma_u\cup \Gamma_b$ and $f\in L^\infty(\Omega)$.
The existence of a unique solution to the problems \eqref{eq:LaplaceResolventOmega} and \eqref{eq:LaplaceResolventDerivativeOmega} in $L^q_H L^p_z(\Omega)$ for such smooth $f$ follows from the properties of the mappings $(\lambda-\Delta)^{-1}$ and $(\lambda-\Delta)^{-1}\partial_i$ in $L^r(\Omega)$ for $\lambda\in \lambda \in \Sigma_{\theta}$
since 
$$v\in W^{2,r}(\Omega)\hookrightarrow L^\infty(\Omega)\hookrightarrow L^q_H L^p_z(\Omega) \quad \hbox {and} \quad w\in W^{1,r}(\Omega)\hookrightarrow L^\infty(\Omega)\hookrightarrow L^q_H L^p_z(\Omega), \quad r >3.$$
%
%
It therefore suffices to prove the estimates \eqref{eq:AnisotropicResolventOmega} and  \eqref{eq:AnisotropicResolventDerivativeOmega}. This is done in the following by localizing the results of Lemma~\ref{LemmaLaplaceResolventFullspace}. 

For this purpose we first make use of the extension operator
		\[
		E=E^\text{even,odd}_{z}\circ E^{\text{per}}_{H}
		\]
where $E^{\text{per}}_{H}$ is the periodic extension operator from $G$ to $\R^2$ in horizontal direction and $E_z^{\text{even,odd}}$ extends from $(-h,0)$ to $(-2h,h)$ in vertical direction via even and odd reflexion at the top and bottom part of the boundary respectively. 

Second, we utilize a family of cut-off-functions $\chi_r\in C^\infty_c(\R^3)$ for $r\in (0,\infty)$ of the form $\chi_r(x,y,z)=\varphi_r(x,y)\psi_r(z)$ where $\varphi_r \in C_c^\infty(\R^2)$ and $\psi_r\in C_c^\infty(\R)$ satisfy
		\begin{align*}
		\begin{array}{rlrl}
		\varphi_r \equiv 1 & \text{ on } [-1/4,5/4]^2, &
		\varphi_r \equiv 0 & \text{ on } \left((-\infty,-r-1/4]\cup [5/4+r,\infty)\right)^2,
		\\
		\psi_r \equiv 1 & \text{ on } [-5h/4,h/4], &
		\psi_r \equiv 0 & \text{ on } (-\infty,-r-5h/4]\cup [h/4+r,\infty),
		\end{array}
		\end{align*}
and there is a constant $M>0$ independent of $r$ such that
		\[
		\lVert \varphi_r\rVert_\infty +\lVert \psi_r\rVert_\infty
		+ r \left(
		\lVert \nabla_H \varphi_r \rVert_\infty +\lVert \partial_z \psi_r \rVert_\infty	
		\right)
		+r^2 \left(
		\lVert \Delta_H \varphi_r\rVert_\infty +\lVert \partial_z^2\psi_r\rVert_\infty
		\right)
		\le M.
		\]
		%
%
Here, we consider $0<4r<3\min\{1,h\}$ which implies that $\varphi_r$ and $\psi_r$ are supported on $(-1,2)$ and $(-2h,h)$ respectively. 
We now define an extension of $v$ from $\Omega$ onto the whole space $\R^3$ via
		\[
		u(x,y,z)=\chi_r(x,y,z) (Ev)(x,y,z)
		\]
for a suitable value of $r$ which we will specify later on. 
If $v$ solves problem \eqref{eq:LaplaceResolventOmega} then $u$ solves the problem
		\[
		\lambda u - \Delta u = F \text{ on } \R^3, \quad 
		F:=\chi_r Ef-2(\nabla \chi_r) \cdot E(\nabla v)-(\Delta \chi_r) Ev.
		\]
Here we made use of the fact that $E$ commutes with derivatives of $v$. 

Note that not only does $F$ have compact support, but we also have $F\in L^\infty(\R^3)$ since we may assume that $f\in L^\infty(\Omega)$ and $v\in W^{1,\infty}(\Omega)$ by the above approximation argument. 
Thus we may now apply Lemma~\ref{LemmaLaplaceResolventFullspace}, and estimate \eqref{eq:AnisotropicResolventFullspace} yields
		\[
		\lvert \lambda \rvert\cdot \lVert u\rVert_{L^q_HL^p_z(\R^3)}
		+\lvert \lambda \rvert^{1/2} \lVert \nabla u\rVert_{L^q_HL^p_z(\R^3)}
		\le C_\theta \lVert F\rVert_{L^q_HL^p_z(\R^3)}.
		\]
To estimate $F$ we use that $\chi_r$ is supported on $(-1,2)^2\times(-2h,h)$, and therefore
		\begin{align*}
		\lVert \chi_r Ef\rVert_{L^q_HL^p_z(\R^3)}
		&\le 27 M^2 \lVert f\rVert_{L^q_HL^p_z(\Omega)} ,
		\\
		\lVert (\nabla\chi_r) \cdot E(\nabla v)\rVert_{L^q_HL^p_z(\R^3)}
		&\le 27 M^2r^{-1}\lVert \nabla v\rVert_{L^q_HL^p_z(\Omega)},
		\\
		\lVert (\Delta\chi_r) Ev\rVert_{L^q_HL^p_z(\R^3)}
		&\le 27 M^2r^{-2}\lVert v\rVert_{L^q_HL^p_z(\Omega)}.
		\end{align*}
		%
Next, we set $r=\eta \lvert \lambda\rvert^{-1/2}$ to obtain
		\begin{align*}
		\lVert F\rVert_{L^q_HL^p_z(\R^3)}
		&\le  27 M^2 \left(
		\lVert f\rVert_{L^q_HL^p_z(\Omega)}
		+2\eta^{-1}\lvert \lambda\rvert^{1/2}\lVert \nabla v\rVert_{L^q_HL^p_z(\Omega)}
		+\eta^{-2} \lvert \lambda\rvert\cdot\lVert v\rVert_{L^q_HL^p_z(\Omega)}
		\right).
		\end{align*}
Now assume that $\eta>0$ is sufficiently large enough such that $54 C_\theta M^2\eta^{-1}<1/2$, $27C_\theta M^2 \eta^{-2}<1/2$ and then assume that $\lambda_0>0$ is large enough such that $4\eta \lambda_0^{-1/2}<3\min\{1,h\}$. This and the fact that $u$ is an extension of $v$ then yields
		\[
		\lvert \lambda \rvert\cdot \lVert v\rVert_{L^q_HL^p_z(\Omega)}
		+\lvert \lambda \rvert^{1/2} \lVert \nabla v\rVert_{L^q_HL^p_z(\Omega)}
		\le 54 C_\theta M^2 \lVert f\rVert_{L^q_HL^p_z(\Omega)} \quad \hbox{for} \quad \lvert \lambda\rvert\ge \lambda_0.
		\]
		%
%

In the case  $q=\infty$, $p\in (2,\infty)$ we obtain the estimate for the full range of $\lambda\in \Sigma_\theta$ by setting $\lambda_1:=\frac{\lambda_0}{\lvert \lambda \rvert}\lambda$ for $0<\lvert \lambda\rvert <\lambda_0$. Then $f\in L^\infty_H L^p_z(\Omega)\hookrightarrow L^p(\Omega)$ yields
		\[
		\lvert \lambda \rvert \cdot \lVert v\rVert_{L^p(\Omega)}
		+\lvert \lambda \rvert^{1/2} \lVert \nabla v \rVert_{L^p(\Omega)}
		+\lVert \Delta v\rVert_{L^p(\Omega)}
		\le C_{\theta,p} \lVert f\rVert_{L^p(\Omega)}
		\]
by \eqref{eq:LpResolventEstimateLaplace} and since $\lambda_1 v-\Delta v=f+(\lambda_1-\lambda)v$ we obtain 
		\[
		\lvert \lambda_1 \rvert \cdot \lVert v\rVert_{L^\infty_H L^p_z(\Omega)}
		+\lvert \lambda_1 \rvert^{1/2} \lvert \nabla v \rvert_{L^\infty_H L^p_z(\Omega)}
		+\lVert \Delta v\rVert_{L^\infty_H L^p_z(\Omega)}
		\le C_{\theta,p} \left(
		\lVert f\rVert_{L^\infty_H L^p_z(\Omega)}
		+\lvert \lambda_1-\lambda\rvert \cdot \lVert v \rVert_{L^\infty_H L^p_z(\Omega)}
		\right)
		\]
where we can further estimate $\lvert \lambda_1-\lambda\rvert<\lambda_0$, and 
$p\in (1,\infty)$ yields 
		\[
		\lVert v \rVert_{L^\infty_H L^p_z(\Omega)}
		\le C_p \lVert v\rVert_{W^{2,p}_H L^p_z(\Omega)}
		\le C_p \lVert v\rVert_{W^{2,p}(\Omega)}
		\le C_p \lVert \Delta v\rVert_{L^p(\Omega)}
		\le C_p\lVert f\rVert_{L^p(\Omega)}
		\le C_p \lVert f\rVert_{L^\infty_H L^p_z(\Omega)}
		\]
where we used $W^{2,p}(G)\hookrightarrow L^{\infty}(G)$ and that $\Delta_p$ is invertible on $L^p(\Omega)$. Since $\lvert \lambda_1\rvert=\lambda_0>\lvert \lambda\rvert$, this yields the desired result for the full range of $\lambda\in \Sigma_\theta$, $\theta\in (0,\pi)$.

If $v$ instead solves problem \eqref{eq:LaplaceResolventDerivativeOmega} with 
$\partial_i=\partial_z$ then $u$ solves the problem
		\[
		\lambda u-\Delta u=G \text{ on } \R^3, \quad 
		G:=\chi_r E(\partial_z f)-2(\nabla \chi_r) \cdot E(\nabla v)-(\Delta\chi_r)Ev.
		\]
We rewrite 
		\[
		-2(\nabla\chi_r) \cdot E(\nabla v)-(\Delta\chi_r)Ev
		=-2 \text{div} (\nabla\chi_r Ev)+(\Delta\chi_r) Ev,
		\quad
		\chi_r E(\partial_z f)=\partial_z (\chi_r s Ef)-(\partial_z\chi_r) s Ef
		\]
where 
		\[
		s(z)=\begin{cases}
		1, & z\in (-2h,0), 
		\\
		-1, & x\in (0,h).
		\end{cases}
		\]
		%
Here, by the density argument above, we may assume $f=0$ on $\Gamma_u\cup \Gamma_b$. 
This yields $u=u_1+u_2$ where
		\begin{align*}
		\begin{array}{rlll}
		\lambda u_1-\Delta u_1&=\partial_z G_1 +\text{div}_H G_2 &\text{ on } \R^3, &
		G_1:=\chi_r s Ef, \quad G_2:=-2 (\nabla \chi_r) Ev,
		\\
		\lambda u_2-\Delta u_2&=G_3 &\text{ on } \R^3, &
		G_3:=-(\partial_z\chi_r) s Ef+(\Delta\chi_r) Ev.
		\end{array}
		\end{align*}
Since $G_i$, $i\in \{1,2,3\}$, are bounded and have compact support, we may apply Lemma~\ref{LemmaLaplaceResolventFullspace} to obtain the estimate
		\[
		\lvert \lambda \rvert^{1/2} \lVert u\rVert_{L^q_H L^p_z(\R^3)}
		\le C_\theta \left(
		\lVert G_1\rVert_{L^q_H L^p_z(\R^3)}
		+\lVert G_2\rVert_{L^q_H L^p_z(\R^3)}
		+\lvert \lambda \rvert^{-1/2} \lVert G_3\rVert_{L^q_H L^p_z(\R^3)}
		\right).
		\]
Proceeding as above we obtain
		\begin{align*}
		\lVert G_1\rVert_{L^q_H L^p_z(\R^3)}
		&\le 27 M^2\lVert f\rVert_{L^q_H L^p_z(\Omega)},
		\\
		\lVert G_2\rVert_{L^q_H L^p_z(\R^3)}
		&\le 54 M^2 \eta^{-1} \lvert\lambda\rvert^{1/2}
		\lVert v\rVert_{L^\infty_H L^p_z(\Omega)},
		\\
		\lVert G_3\rVert_{L^\infty_H L^p_z(\R^3)}
		&\le 27 M^2\eta^{-1}\lvert \lambda\rvert^{1/2} 
		\lVert f\rVert_{L^\infty_H L^p_z(\Omega)}
		+27 M^2 \eta^{-2}\lvert \lambda\rvert\cdot 
		\lVert v\rVert_{L^\infty_H L^p_z(\Omega)}.
		\end{align*}
The above assumptions on $\eta$ and $\lambda_0$ then yield the desired result for $\lvert \lambda\rvert>\lambda_0$. The case $\partial_i\in \{\partial_x,\partial_y\}$ is analogous where for $f\in L^\infty(G;C^\infty_c(-h,0))$ horizontal derivatives are understood in the sense of distributions, and otherwise derivatives can be treated using smooth approximations as above. 

For the case $q=\infty$  and $p\in (2,\infty)$, to extend this estimate to the full range of $\lambda\in \Sigma_\theta$ one proceeds as above to obtain	
		\[
		\lvert \lambda\rvert^{1/2}\cdot \lVert v\rVert_{L^p(\Omega)}
		+ \lVert \nabla v\rVert_{L^p(\Omega)}
		\le C_{\theta,p} \lVert f\rVert_{L^p(\Omega)}
		\]	
from \eqref{eq:LpEstimateDerivativesLaplace},
as well as 
		\[
		\lvert \lambda_1\rvert^{1/2}\cdot \lVert v\rVert_{L^\infty_H L^p_z(\Omega)}
		+ \lVert \nabla v\rVert_{L^\infty_H L^p_z(\Omega)}
		\le C_{\theta}\left(
		\lVert f\rVert_{L^\infty_H L^p_z(\Omega)}
		+ \lvert \lambda_1\rvert^{-1/2} \lvert \lambda_1-\lambda\rvert \cdot 
		\lVert v\rVert_{L^\infty_H L^p_z(\Omega)}
		\right).
		\]
Since we have $\lvert \lambda_1\rvert^{-1/2} \lvert \lambda_1-\lambda\rvert\le \lambda_0^{1/2}$ and $p\in (2,\infty)$ this yields
		\[
		\lVert v\rVert_{L^\infty_H L^p_z(\Omega)}
				\le C_p \lVert v\rVert_{W^{1,p}_H L^p_z(\Omega)}
		\le C_p \lVert v\rVert_{W^{1,p}(\Omega)}
		\le C_p \lVert \nabla v \rVert_{L^p(\Omega)}
		\le C_p \lVert f\rVert_{L^p(\Omega)}
		\le C_p \lVert f\rVert_{L^\infty_H L^p_z(\Omega)},
		\]
where we used the embedding $W^{1,p}(G)\hookrightarrow L^{\infty}(G)$ and 
the Poincar{\'e} inequality $\lVert v\rVert_{L^p(\Omega)} \leq C_p \lVert \nabla v\rVert_{L^p(\Omega)}$ for $v$ with $\restr{v}{\Gamma_b}=0$.
\end{proof}

\begin{remark}\label{rem:AnisotropicLaplaceResolventOmega}
The results of Lemma~\ref{LemmaAnisotropicLaplaceResolventOmega} also hold true if the condition $\restr{\partial_z v}{\Gamma_u}=0$ is replaced by $\restr{v}{\Gamma_u}=0$ or if $L^q_HL^p_z(\Omega)$ is replaced by 
$C_{\text{per}}([0,1]^2;L^p(-h,0))$. For pure Dirichlet boundary conditions one extends by an odd reflexion at both $z=0$ and $z=-h$ replacing $E^\text{even,odd}_{z}$ by $E^\text{odd,odd}_{z}$ and setting $s(z)\equiv 1$ in the proof.
\end{remark}

Since $\Omega=G\times(-h,0)$ is a cylindrical domain the semigroup generated by the Laplacian with the above boundary conditions satisfies
		\[
		e^{t\Delta}(f\otimes g)
		=e^{t\Delta_H}f\otimes e^{t\Delta_z}g, \quad 
		f:G\to \R^2, \quad g:(-h,0)\to \R,
		\]
where $(f\otimes g)(x,y,z):=f(x,y)g(z)$ is an elementary tensor, $\Delta_H:=\partial_x^2+\partial_y^2$ is the Laplacian on $G$ with periodic boundary conditions and 
$\Delta_z$ is defined by
		\[
	\Delta_z v := \partial_z^2 v, \quad D(\Delta_z)=\{ f\in W^{2,p}(-h,0):f(-h)=\partial_z f(0)=0 \}.
		\]
%


We now investigate these operators separately, starting with the vertical one, cf. \cite{DenkHieberPruess2003, Nau2012}.
\begin{lemma}\label{LemmaVerticalSemigroup}
Let $p\in(1,\infty)$. Then the operator $\Delta_z$ 
		%
		%
generates a strongly continuous, exponentially stable, analytic semigroup on $L^p(-h,0)$. 
		%
\end{lemma}

\begin{lemma}\label{LemmaHorizontalSemigroup}
Let $\theta\in(0,\pi/2)$. Then there exists a constant $C_\theta>0$ such that for all $\tau\in\Sigma_\theta$ we have
		\begin{align*}
		\lvert \tau\rvert^{1/2}\lVert \nabla_H e^{\tau\Delta_H}Q f\rVert_{L^\infty(G)}
		\le C_\theta \lVert f\rVert_{L^\infty(G)}, \quad f\in L^\infty(G).
		\end{align*}
\end{lemma}

\begin{remark}
\label{rem:HelmholtzLinfty}
Note that although the two-dimensional Helmholtz projector with periodic boundary conditions $Q$ is unbounded on $L^{\infty}(G)$, the composition $\nabla_H e^{\tau\Delta_H}Q$ defines a bounded operator for $\tau\in \Sigma_\theta$.
\end{remark}

\begin{proof}[Proof of Lemma~\ref{LemmaHorizontalSemigroup}]
Let $Q_{\R^2}$ and $Q$ be the Helmholtz projection on $\R^2$ and $\T^2$, respectively, and $E_H^{\text{per}}$ be the periodic extension operator from $G$ onto $\R^2$. Then $E^{\text{per}}_H Q f=Q_{\R^2} E^{\text{per}}_H f$ for all $f:G\to \R^2$ and
		\begin{align*}
		E_H^{\text{per}}\lvert \tau\rvert^{1/2} \nabla_H e^{\tau\Delta_H}Q f
		=\lvert \tau\rvert^{1/2} \nabla_H e^{\tau\Delta_H}E_H^{\text{per}}Q f
		=\lvert \tau\rvert^{1/2} \nabla_H e^{\tau\Delta_H}Q_{\R^2} E_H^{\text{per}} f.
		\end{align*} 
		%
%
Since $\lVert E_H^{\text{per}}f\rVert_{L^\infty(\R^2)}=\lVert f\rVert_{L^\infty(G)}$ it therefore suffices to consider the operator $\Delta_H$ on the full space $\R^2$.
Recall that $\mathds{1}-Q_{\R^2}$ is given by $(R_jR_k)_{1\le j,k\le 2}$ where $R_j$ is the Riesz transform in the $j$-th direction. 
We therefore investigate the family of Fourier multipliers
		\[
		m_{\tau,j,k,l}(\xi)=
		\begin{cases}
		&\lvert \tau \rvert^{1/2}\xi_l \left(
		\delta_{j,k}-\frac{\xi_j\xi_k}{\lvert \xi\rvert^2}
		\right)
		e^{-\tau\lvert \xi\rvert^2},
		\quad \xi\in\R^2\setminus\{0\},
		\\
		&0, \quad \xi=0,
		\end{cases} \quad \hbox{for} \quad 1\le j,k,l \le 2.
		\]
		%
%
Using the invariance under rescaling and replacing $\xi$ with $\lvert \tau\rvert^{-1/2}\xi$, we may assume that $\tau=e^{i\psi}$ where $\lvert \psi\rvert < \theta$. 
We show that for each of these symbols we have $m=\hat{g}$ for some $g\in L^1(\R^2)$ such that $\lVert g\rVert_{L^1(\R^2)}\le C_\theta$. The desired estimate then follows from Young's inequality. 
Since this family of symbols belongs to $C(\R^2)\cap C^\infty(\R^2\setminus\{0\})$ we verify the Mikhlin condition
		\begin{align}\label{eq:Mikhlin}
		\max_{\lvert \alpha\rvert\le 2}\sup_ {\xi\in\R^2\setminus\{0\}} 
		\lvert \xi\rvert^{\lvert\alpha\rvert+\delta} 
		\lvert D^\alpha m(\xi)\rvert<M<\infty, 
		\end{align}
for some $\delta>0$. 
Elementary calculations using the homogeneity of the first factor show that for an arbitrary multi-index $\alpha\in\N^2$ we have
		\[
		\sup_ {\xi\in\R^2\setminus\{0\}}
		\lvert \xi\rvert^\alpha 
		\left\lvert D^\alpha \frac{\xi_j \xi_k}{\lvert \xi\rvert^2}\right\rvert
		<M_\alpha
		<\infty,
		\quad
		\sup_ {\xi\in\R^2\setminus\{0\}}
		\lvert \xi\rvert^{\alpha+\delta} 
		\left\lvert D^\alpha \xi_l e^{-e^{i\psi}\lvert \xi\rvert^2}\right\rvert
		<M_{\alpha,\delta,\psi}
		\le M_{\alpha,\delta,\theta}
		<\infty
		\]
for $\delta\in(0,1)$ which together with the product rule yield that \eqref{eq:Mikhlin} is satisfied. Analogously we verify the condition
		\begin{align}\label{eq:NotMikhlin}
		\lvert \xi\rvert^{\lvert \alpha\rvert}\lvert D^\alpha m(\xi)\rvert
		\le C_\alpha \lvert \xi\rvert, 
		\quad \lvert \xi\rvert\le 1, \xi\neq 0
		\end{align}
for $0<\lvert \alpha\rvert\le 2$ by noting that
		\begin{align*}
		\lvert \xi\rvert^{\lvert \alpha\rvert} 
		\left\lvert D^\alpha \frac{\xi_j \xi_k \xi_j}{\lvert \xi\rvert^2}\right\rvert 
		\le C_\alpha \lvert \xi\rvert, \quad \lvert \xi\rvert\le 1, \xi\neq 0
		\end{align*}
and 
		\begin{align*}
		\lvert \xi\rvert^{\lvert \alpha\rvert} \left\lvert D^\alpha 
		e^{-e^{i\psi}\lvert \xi\rvert^2}\right\rvert 
		+\lvert \xi\rvert^{\lvert \alpha\rvert} \left\lvert D^\alpha \xi_l 
		e^{-e^{i\psi}\lvert \xi\rvert^2}\right\rvert 
		\le C_{\alpha,\delta,\psi}\le C_{\alpha,\delta,\theta}
		\quad \lvert \xi\rvert\le 1, \xi\neq 0.
		\end{align*}
We now split the symbol into $m=\varphi m+(1-\varphi)m$ where $\varphi\in C^\infty_c(\R^2)$ is a cut-off function satisfying $\varphi(\xi)=1$ for $\lvert \xi\rvert\le 2$. Applying \cite[Lemma 8.2.3 and 8.2.4]{ArendtBattyHieberNeubrander} to the terms $(1-\varphi)m$ and $\varphi m$  respectively then yields the desired results.
\end{proof}

\section{Linear estimates for the hydrostatic Stokes operator: part 1}\label{sec:stokesEasy}

A key element in the proof of our global existence results are the estimates for the hydrostatic Stokes semigroup in $\Xipos$.  
To this end, we prove first estimates in the larger space $\Xip$, where we make use of representation \eqref{eq:Ap}. We thus define the operator $\Aip$ by
	\begin{align*} 
		\Aip v := \Ae v, \quad  D(\Aip)
		=\{
		v\in W^{2,p}_{\text{per}}(\Omega)^2\cap \Xip\colon 
		\restr{\partial_z v}{\Gamma_u}=0, 
		\restr{v}{\Gamma_b}=0, 
		\Ae v\in \Xip
		\}.
	\end{align*}
	It is the aim of this section to prove the following claim.


%

\begin{claim}\label{SemigroupEstimates}
Let $p\in (3,\infty)$. Then 
\\
a) $\Aip$ is the generator of a strongly continuous, analytic semigroup on $\Xip$. 
		%
		%
\\
b) There exist constants $C>0$, $\beta\in\R$ such that for $\partial_i\in \{\partial_x,\partial_y, \partial_z\}$, $t>0$ and $f\in \Xip$ one has that
\begin{align*}
\tag{i} t^{1/2}\lVert \partial_i e^{tA} f\rVert_{L^\infty_H L^p_z(\Omega)}
		\le C e^{\beta t}\lVert f\rVert_{L^\infty_H L^p_z(\Omega)}, \quad t>0, \, f \in \Xip,
\end{align*}
	for $\partial_j\in \{\partial_x,\partial_y\}$
	\begin{align*}
\tag{ii}	t^{1/2}\lVert \partial_j e^{tA}\PP f\rVert_{L^\infty_H L^p_z(\Omega)}
		\le C e^{\beta t}\lVert f\rVert_{L^\infty_H L^p_z(\Omega)}, \quad t>0, \, f \in \Xip, \\
\tag{iii} t^{1/2}\lVert e^{tA}\PP \partial_j f\rVert_{L^\infty_H L^p_z(\Omega)}
		\le C e^{\beta t}\lVert f\rVert_{L^\infty_H L^p_z(\Omega)}, \quad t>0, \, f \in \Xip, \\
\tag{iv} 	t\lVert \partial_i e^{tA}\PP \partial_j f\rVert_{L^\infty_H L^p_z(\Omega)}
		\le C e^{\beta t}\lVert f\rVert_{L^\infty_H L^p_z(\Omega)}, \quad t>0, \, f \in \Xip.			
	\end{align*}
c) $\Xipos$ is an invariant subspace of $\Aip$, and its restriction is $\Aipos$. The semigroup $e^{t\Aip}$ restricts to an exponentially stable, strongly continuous, analytic semigroup of angle $\pi/2$ on $\Xipos$.
\\
d)
Furthermore, for all $v\in \Xipos$ 
		\begin{align*} 
		\lim_{t \to 0+} t^{1/2}\lVert \nabla e^{tA}v\rVert_{L^\infty_H L^p_z(\Omega)}=0.
		\end{align*}
\end{claim}

In order to  solve equation \eqref{eq:hydrostaticStokes} in $\Xipos$, we collect first several facts concerning the  corresponding theory in $\lpso$. To this end,  
let $p\in (1,\infty)$, and define $\Apos\colon D(\Apos) \rightarrow \lpso$ by 
\begin{align*}
		\Apos v :=\PP\Delta v, \quad  D(\Apos)=\{
		v\in W^{2,p}_{\text{per}}(\Omega)^2: 
		\text{div}_H \overline{v}=0,
		\restr{\partial_z v}{\Gamma_u}=0, 
		\restr{v}{\Gamma_b}=0.
		\}.
\end{align*}
Consider furthermore $\Ap\colon D(A_p)\to L^p(\Omega)^2$ defined by
		\begin{align*}
		\Ap v:=\Delta_p v+B v, \quad  D(A_p):=D(\Delta_p)^2,
		\quad 
		Bv:=\frac{1}{h}(1-Q)\restr{\partial_z v}{\Gamma_b}, 
		\end{align*}
where $\Delta_p$ denotes the Laplacian in $L^p(\Omega)^2$ as in the last section. By \cite{GGHHK17}, the operator $\Ap$ is an extension of $\Apos$. The idea is that the pressure term may be recovered by applying the vertical average and horizontal divergence to \eqref{eq:hydrostaticStokes}, yielding
		\begin{align}\label{eq:LinearPressureWeak}
		\Delta_H \pi = \text{div}_H \overline{f}
		-\text{div}_H\frac{1}{h}\restr{\partial_z v}{\Gamma_b}, 
		\end{align}
or equivalently since $1-Q$ agrees with $\nabla_H (-\Delta_H)^{-1}\text{div}_H$ one has $\nabla_H \pi=(1-Q)\overline{f}-Bv$. 

Note that the following inclusions hold
\begin{align}\label{eq:domaininclusion}
\Aip \subset \Ap \quad \hbox{and} \quad \Aipos \subset \Apos,
\end{align}
and that $e^{t\Apos}$, $e^{tA_p}$, $e^{t\Aip}$ and $e^{t\Aipos}$ are consistent semigroups. 

%

\begin{proof}[Proof of Claim~\ref{SemigroupEstimates}]
%

%
Let $\lambda_0>0$ with $\lambda_0\in \rho(A_p)$, $\theta\in (0,\pi/2)$, and 
		%
		\[
		\lambda \in
		\Sigma_{\theta+\pi/2} \cap B_{\lambda_0}(0)^c\subset \rho(A_p).
		\]
		%
By \eqref{eq:domaininclusion} it follows that $\lambda-\Aip$ is injective for $\lambda \in \rho(A_p)$ and likewise $\lambda-\Aipos$ is injective for $\lambda \in \rho(A^\os_p)$.
Since $\Xip\hookrightarrow L^p(\Omega)^2$ the existence of a unique $v\in D(A_p)$ for $p\in (1,\infty)$ follows from the $L^p$-theory for $A_p$, cf. \cite{GGHHK17},
and since $W^{2,p}_{\text{per}}(\Omega)^2 \hookrightarrow \Xip$ for $p\in (3/2,\infty)$ it follows that $v\in D(\Aip)$. Since 
$(A_p-\lambda)^{-1}$ further leaves $\lpso$ invariant, $f\in \Xipos$ implies $v\in D(\Aipos)$.
Hence, 
		\begin{align}\label{eq: resolvent inclusions}
		\rho(\Ap)\subset \rho(\Aip) \quad \text{ and } \quad 
		\rho(\Apos)\subset \rho(\Aipos).
		\end{align}
In particular the resolvent sets are non-empty and thus the operators are closed.

 Since the semigroup estimates follow from resolvent estimates by arguments involving the inverse Laplace transform, it now remains to prove suitable resolvent estimates in $\Xip$. To this end we observe first that $v=(\lambda-\Aip)^{-1}f$ is equivalent to
\begin{align}\label{eq:repv}
v=(\lambda-\Delta_p)^{-1}(f+Bv),
\end{align}
and second, using the fact that $Q$ is continuous on $C_{per}^{0,\alpha}([0,1]^2)$ for $\alpha\in (0,1)$, that
		\begin{align*}
		\lVert Bv\rVert _{L^\infty_H L^p_z(\Omega)}
		\le h^{1/p} \lVert Bv\rVert _{L^\infty(\Omega)}
		\le h^{1/p} \lVert Bv\rVert_{C^{0,\alpha}({[0,1]^2)}}
		\le C \lVert \restr{\partial_z v}{\Gamma_b}\rVert_{C^{0,\alpha}([0,1]^2)}
		\le C \lVert v\rVert_{C^{1,\alpha}(\overline{\Omega})}.
		\end{align*}
Assuming $p\in (3,\infty)$ we have $W^{2,p}(\Omega)\hookrightarrow C^{1,\alpha}(\overline{\Omega})$ for some $\alpha=\alpha_p\in (0,1-3/p)$. 
Using the resolvent estimate for $\Ap$ in $L^p(\Omega)^2$ we obtain
		\[
		\lVert v\rVert_{C^{1,\alpha}(\overline{\Omega})}
		\le C_p \lVert v\rVert_{W^{2,p}(\Omega)}
		\le C_p \left( 
		\lVert v\rVert_{L^p(\Omega)}+\lVert A v\rVert_{L^p(\Omega)}
		\right)
		\le C_p (1+\lvert \lambda\rvert^{-1}) \lVert f\rVert_{L^p(\Omega)}.
		\]
This and $\lvert \lambda\rvert>\lambda_0$ yield $\lVert Bv\rVert _{L^\infty_H L^p_z(\Omega)}\le C_p (1+\lambda_0^{-1})\lVert f\rVert_{L^\infty_H L^p_z(\Omega)}$.
So, using Lemma~\ref{LemmaAnisotropicLaplaceResolventOmega} we obtain
		\begin{align} \label{eq1: conclusion of resolvent estimate}
		\lvert \lambda\rvert\cdot \lVert v\rVert_{L^\infty_H L^p_z(\Omega)}
		+\lvert \lambda\rvert^{1/2} \lVert \nabla v\rVert_{L^\infty_H L^p_z(\Omega)}
		+ \lVert \Aip v\rVert_{L^\infty_H L^p_z(\Omega)}
		&\le C_{\theta,p,\lambda_0}\lVert f\rVert_{L^\infty_H L^p_z(\Omega)},
		\end{align}
where we used that for $\lambda$ as above and $p\in (3,\infty)$ one has
		\begin{align*}
		\lVert \Aip v\rVert_{L^\infty_H L^p_z(\Omega)}
		\le \lVert \Delta v\rVert_{L^\infty_H L^p_z(\Omega)}
		+\lVert Bv\rVert_{L^\infty_H L^p_z(\Omega)}
		\le C_{\theta,p,\lambda_0}\lVert f\rVert_{L^\infty_H L^p_z(\Omega)}.
		\end{align*}

Note that if one instead considers $f\in \Xipos$, then $\lambda_0>0$ can be taken to be arbitrarily small and $\theta$ arbitrarily close to $\pi/2$ by \cite[Theorem 3.1]{GGHHK17}.
%
Since $0\in \rho(\Apos)\subset \rho(\Aipos)$, compare \cite[Theorem 3.1]{HieberKashiwabara2015}
 and \eqref{eq: resolvent inclusions} 
it follows that the spectral bound
		\[
		\beta:=\sup\{
		\text{Re}(\lambda): \lambda\in \sigma(\Aipos)
		\}
		\]
is negative implying 
exponential decay, 
and estimate \eqref{eq1: conclusion of resolvent estimate} is valid for all $\lambda\in\Sigma_\theta$, $\theta\in(0,\pi)$ and $f\in \Xipos$. 

To verify that $D(\Aip)$ and $D(\Aipos)$ are dense in $\Xip$ and $\Xipos$ respectively, observe that the space
		\[
		C^\infty_{\text{per}}([0,1]^2;C^\infty_c((-h,0)))^2
		\]
is contained in $D(\Aip)$ and dense in $\Xip$, so the semigroup generated by $\Aip$ is strongly continuous on $\Xip$. Since it leaves $\lpso$ invariant, the restriction of the semigroup on $\Xip\cap \lpso=\Xipos$ is strongly continuous as well and generated by the restriction of $\Aip$ onto $D(\Aip)\cap \lpso=D(\Aipos)$, i.e. $\Aipos$, which is therefore densely defined on $\Xipos$. Thus we have proven $a)$, $c)$ and estimate $(i)$ in $b)$.

To prove the remaining semigroup estimates in $b)$ we consider the corresponding resolvent estimates. 
Since $\Xip\hookrightarrow L^p(\Omega)^2$ and $\PP$ is bounded on $L^p(\Omega)^2$ the existence of 
		\[
		v:=(\lambda-\Apos)^{-1}\PP f\in D(\Apos)
		\hookrightarrow W^{2,p}_{\text{per}}(\Omega)^2
		\hookrightarrow \Xip
		\]
for $f\in \Xip$ follows from 
the $L^p$-theory for $\Apos$,
and it suffices to extend 
the $L^p$-estimate 
	\begin{align}\label{eq:LpEstimateDerivatives}
	\lvert \lambda \rvert^{1/2}\lVert \partial_i(\lambda-\Apos)^{-1} f\rVert_{L^p(\Omega)}
	+\lvert \lambda \rvert^{1/2}\lVert 
	(\lambda-\Apos)^{-1}\partial_i f\rVert_{L^p(\Omega)}
	\le C_{\theta,p} \lVert f\rVert_{L^p(\Omega)}, \quad f\in \lpso,
	\end{align}
	where $\partial_i\in\{\partial_x,\partial_y\}$, 
	$\theta \in (0,\pi)$, $C_{\theta,p}>0$,
to $\Xip$, i.e. to prove the estimate
		\begin{align} \label{eq: resolvent estimate for horizontal derivative}
		\lvert \lambda \rvert^{1/2} \lVert \nabla_H v\rVert_{L^\infty_H L^p_z(\Omega)}
		\le C_{\theta,p}\lVert f\rVert_{L^\infty_H L^p_z(\Omega)}, 
		\quad \lambda \in \Sigma_\theta.
		\end{align}
Recall that $\PP f=f-(1-Q)\overline{f}=\tilde{f}+Q\overline{f}$, 
		%
and that if $f\in \Xip$ then $\overline{f}\in C_{\text{per}}([0,1]^2)^2$ satisfies $\lVert \overline{f}\rVert_{\infty}\le C \lVert f\rVert_{L^\infty_H L^p_z(\Omega)}$ for any $p\in[1,\infty]$.
Using \eqref{eq:repv} we rewrite
		\[
		v=(\lambda-\Aipos)^{-1}\PP f=(\lambda-\Delta)^{-1}(\tilde{f}+Bv+Q\overline{f}),
		\]
and since the term $\tilde{f}+Bv$ can be dealt with as before, it suffices to show the estimate
		\begin{align}
		\lvert \lambda\rvert^{1/2} \lVert 
		\nabla_H (\lambda-\Delta)^{-1}Q\overline{f}
		\rVert_{L^\infty_H L^p_z(\Omega)}
		\le C_\theta \lVert \overline{f}\rVert_{L^\infty(G)}.
		\end{align}
Since $Q\overline{f}$ does not depend on $z$ we can write $Q\overline{f}=Q\overline{f}\otimes 1$, and so for $\lambda=\lvert \lambda\rvert e^{i\psi}$ with $\psi\in(-\pi/2+\varepsilon,\pi/2-\varepsilon)$ for small $\varepsilon>0$ we have
		\begin{align*}
		\lvert \lambda\rvert^{1/2}\nabla_H(\lambda-\Delta)^{-1}
		\left(Q\overline{f}\otimes 1\right)
		=\lvert \lambda\rvert^{1/2}\int_0^\infty e^{-\lambda t}
		\left(
		\nabla_H e^{t\Delta_H}Q\overline{f}\otimes e^{t\Delta_z}1
		\right)
		\,dt,
		\end{align*}
where $e^{t\Delta_z}$ denotes the semigroup from Lemma~\ref{LemmaVerticalSemigroup}. Applying the estimates in Lemma~\ref{LemmaHorizontalSemigroup} and \ref{LemmaVerticalSemigroup} yields
		\begin{align*}
		\lvert \lambda\rvert^{1/2}\lVert 
		\nabla_H(\lambda-\Delta)^{-1}\left(Q\overline{f}\otimes 1\right)
		\rVert_{L^\infty_H L^p_z(\Omega)}
		&\le
		\lvert \lambda\rvert^{1/2}\int_0^\infty e^{-\lambda t}
		\lVert \nabla_H e^{t\Delta_H}Q\overline{f}\rVert_{L^\infty(G)}
		\lVert e^{t\Delta_z}1\rVert_{L^p(-h,0)}
		\,dt
		\\
		&\le C \lvert \lambda\rvert^{1/2}
		\left(
		\int_0^\infty e^{-\lvert\lambda\rvert\cos(\psi) t}t^{-1/2}\,dt 
		\right)
		\lVert \overline{f}\rVert_{L^\infty(G)}
		\\
		&\le C \frac{\sqrt{\pi}}{\sqrt{\cos(\pi/2-\varepsilon)}}
		\lVert \overline{f}\rVert_{L^\infty(G)}.
		\end{align*}
To include the full range of angles $\psi$ one simply replaces $\Delta_H$ and $\Delta_z$ with $e^{i\theta}\Delta_H$ and $e^{i\theta}\Delta_z$ respectively where $\theta\in(-\pi/2,\pi/2)$ is a suitable angle.

Since an elementary calculation shows that $\nabla_H$ commutes with $\Aip$ and $\PP$ we obtain
		\[
		\partial_i (\lambda-\Aip)^{-1} f
		=(\lambda-\Aip)^{-1}\partial_i f, \quad
		\partial_i (\lambda-\Aip)^{-1}\PP f
		=(\lambda-\Aip)^{-1}\PP\partial_i f
		\]
for horizontal derivatives $\partial_i\in\{\partial_x,\partial_y\}$ and $f\in C^\infty_{\text{per}}([0,1]^2;C^\infty_c[-h,0])^2$. Note that for any $v\in W^{2,p}_{\text{per}}(\Omega)$ the horizontal derivatives $\partial_x v$ and $\partial_y v$ are periodic on $\Gamma_l$ as well. This yields suitable estimates for the right-hand sides. 

To verify $d)$, we first make use of the density of the domains of the generators. So, let $\varepsilon>0$ and $v'\in D(\Aipos)$ such that $\lVert v-v'\rVert_{L^\infty_H L^p_z(\Omega)}<\varepsilon/2C_0$. By $b)$ $(i)$ we have
\[
t^{1/2}\lVert \nabla e^{t\Aip} v\rVert_{L^\infty_H L^p_z(\Omega)}
\le C_0 \lVert v\rVert_{L^\infty_H L^p_z(\Omega)}
\]
for all $v\in \Xip$ and $t>0$. Then
\[
t^{1/2}\lVert \nabla e^{t\Aip}v\rVert_{L^\infty_H L^p_z(\Omega)}
\le \frac{\varepsilon}{2}+t^{1/2}\lVert \nabla e^{t\Aip}v'\rVert_{L^\infty_H L^p_z(\Omega)}
\] 
and we can further estimate
\[
\lVert \nabla e^{t\Aip}v'\rVert_{L^\infty_H L^p_z(\Omega)}
\le h^{1/p}\lVert e^{t\Aip}v'\rVert_{C^1(\overline{\Omega})}
\le C_p \lVert e^{t\Aip}v'\rVert_{D(\Apos)}.
\]
This and the invertibility of $\Apos$ on $\lpso$ yield
\begin{align*}
t^{1/2}\lVert \nabla e^{t \Apos}v'\rVert_{\lpso}
\le C_p t^{1/2} \lVert \Apos e^{t\Apos}v'\rVert_{\lpso}
=C_p t^{1/2}\lVert  e^{t\Apos}\Apos v'\rVert_{\lpso} \leq C_p t^{1/2}\lVert \Apos v'\rVert_{\lpso}
\end{align*}
and since $\Apos v'\in \lpso$ the claim follows.
\end{proof}

\section{Linear estimates for the hydrostatic Stokes operator: part 2}\label{sec:stokesHard}

This section is devoted to prove that the estimates of Claim~\ref{SemigroupEstimates} in the case of vertical derivatives, i.e.
that the estimates (ii), (iii) and (iv) in Claim~\ref{SemigroupEstimates} are valid even for $\partial_j=\partial_z$.
\begin{claim}\label{SemigroupEstimatesHard}
Under the assumptions of Claim~\ref{SemigroupEstimates} 
there exist constants $C>0$ and $\beta\in \R$ such that

	\begin{align}\label{eq:FirstVerticalEstimate}
	t^{1/2}\lVert \partial_z e^{tA}\mathbb{P}f\rVert_{L^\infty_H L^p_z(\Omega)}
	&\le C e^{t\beta}\lVert f\rVert_{L^\infty_H L^p_z (\Omega)}, \\
	\label{eq:SecondVerticalEstimate}
	t^{1/2}\lVert e^{tA}\mathbb{P}\partial_zf\rVert_{L^\infty_H L^p_z(\Omega)}
	&\le C e^{t\beta}\lVert f\rVert_{L^\infty_H L^p_z (\Omega)}, \\
	\label{eq:ThirdVerticalEstimate}
	t\lVert \partial_i e^{tA}\PP \partial_j f\rVert_{L^\infty_H L^p_z(\Omega)} 
	&\le C e^{\beta t}\lVert f\rVert_{L^\infty_H L^p_z(\Omega)},
		\end{align}		
where $\partial_i, \partial_j\in \{\partial_x,\partial_y,\partial_z\}$,
for all $t>0$ and $f\in \Xip$. 
\end{claim}

As in the last section, these semigroup estimates follow from suitable resolvent estimates and 
standard arguments involving the inverse Laplace transform.  

%

Before investigating the estimate for $\partial_z (\lambda-\Aip)\mathbb{P}$ we present an anisotropic version of an interpolation inequality.
We use the notation $(x,y,z) =: (x', z)$ and let $B(x_0'; r) = \{ x' \in \mathbb R^2 : |x' - x_0'| < r \}$ denote a disk in $\mathbb R^2$.

\begin{lemma} \label{lem: anisotropic interpolation inequality}
Let $p \in (2, \infty)$, $q \in [1,\infty]$, $r>0$, and $x_0' \in \mathbb R^2$.
Then, for $v \in W^{1,p}(B(x_0'; r); L^q_z)$, $L^q_z=L^q(-h,0)$ we have
		\begin{equation*}
		\|v\|_{L^\infty(B(x_0'; r); L^q_z)} 
		\le C r^{-2/p} (
		\|v\|_{L^p(B(x_0'; r); L^q_z)} + r\|\nabla_Hv\|_{L^p(B(x_0'; r); L^q_z)}
		),
		\end{equation*}
where the constant $C = C_{\Omega, p, q}>0$ is independent of $r$ and $x_0'$.
\end{lemma}

\begin{proof}
We put $w(x') := (\int_{-h}^0 |v(x', z)|^q \, dz)^{1/q}$ and apply a two-dimensional interpolation inequality, compare \cite[Lemma 3.1.4]{Lunardi1995} to have
		\begin{equation} \label{eq1: prf of anisotropic interpolation inequality}
		\|w\|_{L^\infty(B(x_0'; r))} 
		\le C r^{-2/p} (
		\|w\|_{L^p(B(x_0'; r))} + r\|\nabla_Hw\|_{L^p(B(x_0'; r))}
		).
		\end{equation}
One sees that $\|w\|_{L^p(B(x_0'; r))} = \|v\|_{L^p(B(x_0'; r); L^q_z)}$.
To estimate the second term we compute $\partial_i w$ for $\partial_i\in\{\partial_x,\partial_y\}$ as follows:
		\begin{align*}
		\partial_i w(x') 
		= \left( \int_{-h}^0 |v(x', z)|^q \, dz \right)^{1/q - 1} 
		\int_{-h}^0 |v(x', z)|^{q - 2} (\partial_i v(x', z) \cdot v(x', z)) \, dz.
		\end{align*}
Using H{\"o}lder's inquality we obtain
		\begin{align*}
		|\partial_i w(x')| 
		\le \left( \int_{-h}^0 |v(x', z)|^q \, dz \right)^{1/q-1} 
		\int_{-h}^0 |v(x', z)|^{q - 1} |\partial_i v(x', z)| \, dz 
		\le \left( \int_{-h}^0 |\partial_i v(x', z)|^q \, dz \right)^{1/q}
		\end{align*}
and substituting this into \eqref{eq1: prf of anisotropic interpolation inequality} proves the estimate for $q<\infty$.
The case $q = \infty$ is a straightforward result of \eqref{eq1: prf of anisotropic interpolation inequality}.
\end{proof}
%
%
%
%

It is well known that $1-Q=-\nabla_H (-\Delta_H)^{-1} \mathrm{div}_H=\nabla_H \Delta_H^{-1} \mathrm{div}_H$ with periodic boundary conditions is a singular integral operator which fails to be bounded in $L^\infty(G)^2$.
However, if one allows for a logarithmic (and therefore divergent) factor, some $L^\infty$-type estimate are still available.
In this spirit we give a local $L^p$-estimate for the operator $\nabla_H (-\Delta_H)^{-1} \mathrm{div}_H$ corresponding to the scale of the $L^\infty$-norm.

\begin{proposition} \label{prop: nearly-optimal estimate for laplacian}
Let $p \in (1,\infty)$, $x_0' \in G$. Then there exists $r_0>0$ such that for all $r\in (0,r_0)$ the weak solution of
		\begin{equation} \label{eq: Laplace equation}
		\Delta_H \pi = \mathrm{div}_H F 
		\quad\text{in}\quad G, \qquad 
		\pi|_{\partial G}: \text{ periodic}, 
		\qquad \int_G \pi \, dx' = 0,
		\end{equation}
for $F \in L^\infty(G)^2$ satisfies
		\begin{equation*}
		\|\nabla_H \pi\|_{L^p(B(x_0'; r))} 
		\le C r^{2/p} (1 + |\log r|) \|F\|_{L^\infty(G)}.
		\end{equation*}
Here the constant $C = C_{G, p}>0$ is independent of $x_0'$ and $r$.
\end{proposition}


\begin{proof}
By applying a periodic extension we may assume that \eqref{eq: Laplace equation} holds in a larger square $G' := (-2, 3)^2$. We choose $r_0<1/8$ to obtain $B(x_0'; 4r_0) \subset (-1/2,3/2)^2$ and utilize two cut-off functions $\omega,\theta \in C^\infty_c(\mathbb R^2)$, $\theta=\theta_r$, satisfying the following properties:
		\begin{align*}
		\begin{array}{lll}
		\omega\equiv1 \text{ on } [-1, 2]^2, & 
		\mathrm{supp}\,(\omega) \subset G', &
		\|\nabla_H^k \omega\|_{L^\infty(\mathbb R^2)} 
		\le C,
		\\
		\theta\equiv1 \text{ on } B(x_0'; 2r), & 
		\mathrm{supp}\,(\theta) \subset B(x_0'; 4r), &
		\|\nabla_H^k \theta\|_{L^\infty(\mathbb R^2)} 
		\le Cr^{-k}
		\end{array}
		\end{align*}
for $k=0,1,2$; compare the proof of Lemma~\ref{LemmaAnisotropicLaplaceResolventOmega}.
From \eqref{eq: Laplace equation} we see that $\omega \pi$ satisfies
		\begin{equation*}
		\Delta_H (\omega \pi) 
		= \mathrm{div}_H (\omega F) 
		- \nabla_H\omega\cdot F 
		+ 2\mathrm{div}_H\big( (\nabla_H\omega)\pi \big) 
		- (\Delta_H\omega)\pi 
		\quad\text{in}\quad \mathbb R^2.
		\end{equation*}
Then, letting $\Psi(x', y') := \frac1{2\pi} \log|x' - y'|$ be the Green's function for the Laplacian in $\mathbb R^2$, we obtain
		\begin{align*}
		(\omega \pi)(x') 
		= - \int_{\mathbb R^2} (\nabla_{y'} \Psi)(x', y') \cdot 
		\big[ 
		\omega F + 2(\nabla_{y'}\omega)\pi 
		\big](y') \, dy'
		- \int_{\mathbb R^2} \Psi(x', y') \big[
		(\nabla_H\omega)\cdot F + (\Delta_H\omega)\pi 
		\big](y') \, dy'.
		\end{align*}
Therefore, for $x' \in B(x_0'; r)$ we have the representation
		\begin{align*}
		\nabla_H \pi(x') 
		= & - \int_{\mathbb R^2} (\nabla_{x'}\nabla_{y'} \Psi)(x', y') \big[
		\omega F + 2(\nabla_{y'}\omega)\pi 
		\big](y') \, dy' 
		- \int_{\mathbb R^2} (\nabla_{x'}\Psi)(x', y') \big[
		(\nabla_H\omega)\cdot F + (\Delta_H\omega)\pi 
		\big](y') \, dy' 
		\\
		= & - \int_{\mathbb R^2} (\nabla_{x'}\nabla_{y'} \Psi)(x', y') \big[
		\theta F + \omega(1 - \theta) F + 2(\nabla_{y'}\omega)\pi 
		\big](y') \, dy' 
		\\
		& - \int_{\mathbb R^2} (\nabla_{x'}\Psi)(x', y') \big[
		(\nabla_H\omega)\cdot F + (\Delta_H\omega)\pi 
		\big](y') \, dy' 
		\\
		=: &\,\Pi_1(x') + \Pi_2(x') + \Pi_3(x') + \Pi_4(x') + \Pi_5(x')
		\end{align*}
where in the second step we used $\omega \theta=\theta$.		
We derive $L^p(B(x_0'; r))$-estimates for each of the above terms as follows:
By the Calder\'on--Zygmund inequality we have
		\begin{equation*}
		\|\Pi_1\|_{L^p(\mathbb R^2)} 
		\le C\|\theta F\|_{L^p(\mathbb R^2)} 
		\le C \|\theta\|_{L^p(\mathbb R^2)} \|F\|_{L^\infty(G')} 
		\le C r^{2/p} \|F\|_{L^\infty(G)}.
		\end{equation*}
		%
For the second term note that we have $|\nabla_{x'}\nabla_{y'} \Psi(x', y')| \le C|x' - y'|^{-2}$ and
		\[
		\mathrm{supp}\,(\omega(1-\theta)) 
		=\mathrm{supp}\,(\omega-\theta)
		\subset \mathrm{supp}(\omega)\setminus B(x'_0;2r)
		\]
yields $\mathrm{supp}\,(\omega(1-\theta)) \subset \{r \le |x'-y'| \le 4\}$ and therefore
		\begin{align*}
		\|\Pi_2\|_{L^p(B(x_0'; r))} 
		&\le \lVert 1\rVert_{L^p(B(x'_0;r))} 
		 \left(
		\sup_{x' \in B(x_0'; r)} \int_{r \le |x' - y'|\le 4} C|x' - y'|^{-2} \, dy' 
		\right)
		\|\omega(1 - \theta) F\|_{L^\infty(G')}
		\\
		&\le C r^{2/p} (1 + |\log r|) \|F\|_{L^\infty(G)}.
		\end{align*}
		%
The condition $\mathrm{supp}\,(\nabla_H\omega)\subset G'\setminus [-1,2]$ yields
		\begin{align*}
		\|\Pi_3\|_{L^p(B(x_0'; r))} 
		&\le \lVert 1\rVert_{L^p(B(x'_0;r))}  
		\left(
		\sup_{1/2 \le |x' - y'| \le 3} C|x' - y'|^{-2} 
		\right)
		\|2 (\nabla_H\omega) \pi\|_{L^1(G')} 
		\le C r^{2/p} \|\pi\|_{L^1(G)}.
		\end{align*}
It follows from Poincar\'e's inequality and the $L^2$-theory for 
\eqref{eq: Laplace equation} that
		\begin{equation*}
		\|\pi\|_{L^1(G')} 
		\le C\|\pi\|_{L^2(G)} 
		\le C\|\nabla_H\pi\|_{L^2(G)} 
		\le C\|F\|_{L^2(G)} 
		\le C\|F\|_{L^\infty(G)}
		\end{equation*}
and therefore $\|\Pi_3\|_{L^p(B(x_0'; r))} \le C r^{2/p} \|F\|_{L^\infty(G)}$.
		%
Similarly to $\Pi_3$, we have
		\begin{align*}
		\|\Pi_4 + \Pi_5\|_{L^p(B(x_0'; r))} 
		&\le \lVert 1\rVert_{L^p(B(x'_0;r))}  \Big( 
		\sup_{1/2 \le |x' - y'| \le 3} C|x' - y'|^{-1} 
		\Big) 
		(\|F\|_{L^1(G)} + \|\pi\|_{L^1(G)}) 
		\\
		&\le C r^{2/p} \|F\|_{L^\infty(G)}.
		\end{align*}
Combining these estimates yields the desired estimate.
\end{proof}

\begin{remark}
Note that the Calder\'on--Zygmund inequality we have used to estimate $\Pi_1$ does not hold for $p\in \{1,\infty\}$ while the arguments of Section~\ref{sec:stokesEasy} can be adapted to cover the case $p=\infty$.
\end{remark}

We now turn to prove the estimate 
$\lvert \lambda\rvert^{1/2}\lVert \partial_z (\lambda-\Aip)^{-1}\mathbb{P}\rVert_{L^\infty_H L^p_z(\Omega)}
\le C_{\theta, p, \lambda_0} \lVert f\rVert_{L^\infty_H L^p_z(\Omega)}$ 
for $\lambda\in\Sigma_{\theta}$, $|\lambda|>\lambda_0$ and for $\theta\in(0,\pi)$, $p>3$.
For this purpose we observe that the solution $v$ to the resolvent problem
		\begin{align*}
		\lambda v-Av=\PP f 
		\quad \text{on} \quad  \Omega
		\end{align*}
with boundary conditions \eqref{eq:bc} is decomposed as $v = v_1 + v_2$, where $(v_1, \pi_1)$ and $(v_1, \pi_1)$ solve
		\begin{equation} \label{eq: problem for v1}
		\lambda v_1 - \Delta v_1 + \nabla_H \pi_1 
		= f \text{ on }\Omega, 
		\quad
		\Delta_H\, \pi_1 = -h^{-1} \mathrm{div}_H(\partial_z v|_{\Gamma_b}) 
		\text{ on } G, 
		\end{equation}
and
		\begin{equation} \label{eq: problem for v2}
		\lambda v_2 - \Delta v_2 + \nabla_H \pi_2 
		= 0 \text{ on } \Omega, 
		\quad
		\Delta_H\, \pi_2 = \mathrm{div}_H\, \bar f 
		\text{ on } G, 
		\end{equation}
respectively, both equipped with the boundary conditions \eqref{eq:bc} and periodic boundary conditions for $\pi_i$ on $\partial G$, as $\pi:=\pi_1+\pi_2$ satisfies \eqref{eq:LinearPressureWeak}.
Since \eqref{eq: problem for v1} is equivalent to $v_1=(\lambda-\Delta)^{-1}(f+Bv)$ we obtain 
		\begin{align}\label{eq:voneestimate}
		|\lambda|^{1/2} \|\partial_z v_1\|_{L^\infty_H L^p_z(\Omega)} 
		\le |\lambda|^{1/2} \|\nabla v_1\|_{L^\infty_H L^p_z(\Omega)} 
		\le C_{\theta, p,\lambda_0} \|f\|_{L^\infty_H L^p_z(\Omega)}
		\end{align}
for $\lvert \lambda\rvert>\lambda_0$ by the same argument used to derive \eqref{eq1: conclusion of resolvent estimate}. 
This, $\nabla_H v_2=\nabla_H v-\nabla_H v_1$, and estimate \eqref{eq: resolvent estimate for horizontal derivative} yield
		\begin{align}\label{eq: horizontal estimate for v two}
		\lvert \lambda \rvert^{1/2} \lVert \nabla_H v_2\rVert_{L^\infty_H L^p_z(\Omega)}
		\le C_{\theta,p,\lambda_0}\lVert f\rVert_{L^\infty_H L^p_z(\Omega)}, 
		\quad \lambda \in \Sigma_\theta.
		\end{align} 
In order to prove estimate~\eqref{eq:FirstVerticalEstimate} it thus remains to establish the following.

\begin{proposition}\label{vtwoestimate}
Let $p \in (3,\infty)$ and $\theta\in (0,\pi)$. Then there exists constants $\lambda_0>0$ and $C_{\theta,p,\lambda_0}>0$ such that for all $\lambda \in \Sigma_\theta$ with $\lvert \lambda\rvert>\lambda_0$ and $f \in \Xip$ the solution $v_2$ of \eqref{eq: problem for v2} satisfies
		\begin{equation*}
		|\lambda|^{1/2} \|\partial_z v_2\|_{L^\infty_H L^p_z(\Omega)} 
		\le C_{\theta,p,\lambda_0} \|f\|_{L^\infty_H L^p_z(\Omega)}.
		\end{equation*}
\end{proposition}

\begin{remark}\label{FullRangeRemark}
	The estimate 
	\[
	\lvert \lambda \rvert^{1/2}
	\lVert \partial_z (\lambda-\Aip)^{-1}\mathbb{P}f\rVert_{L^\infty_H L^p_z(\Omega)}
	\le C_{\theta,p}\lVert f\rVert_{L^\infty_H L^p_z(\Omega)},
	\quad f\in \Xip
	\]
	actually holds for the full range of $\lambda\in\Sigma_\theta$, $\theta\in (0,\pi)$, i.e. one can take $\lambda_0=0$. This is obtained by using that $\mathbb{P}f\in \lpso$ yields $v:=(\lambda-\Aip)^{-1}\mathbb{P}f\in D(\Apos)$ and therefore
	\[
	\lVert v\rVert_{L^\infty_H L^p_z(\Omega)}
	\le C_p \lVert v\rVert_{W^{2,p}(\Omega)}
	\le C_p \lVert Av\rVert_{\lpso}
	\le C_p \lVert \mathbb{P}f\rVert_{\lpso}
	\le C_p\lVert f\rVert_{\lpso}
	\le C_p \lVert f\rVert_{L^\infty_H L^p_z(\Omega)},
	\]
	so the same argument as in the proof of Lemma~\ref{LemmaAnisotropicLaplaceResolventOmega} applies. 
\end{remark}

%
%

\begin{proof}[Proof of Proposition \ref{vtwoestimate}]
We will simply write $(v,\pi)$ instead of $(v_2,\pi_2)$ for the solution of \eqref{eq: problem for v2}.
By applying a periodic extension in the horizontal variables we may assume that \eqref{eq: problem for v2} holds in a larger domain allowing us to replace $\Omega$ and $G$ by $\Omega' := G'\times(-h,0)$ and $G' := (-2,3)^2$ respectively. We decompose the boundary of $\Omega'$ into $\Gamma_u'=G'\times \{0\}$, $\Gamma_l' := \partial G'\times[-h,0]$ and $\Gamma'_b=G\times\{-h\}$.
For simplicity we continue to denote the periodic extensions of $v$, $\pi$ and $f$ in the same manner.

Let $\eta>1$ be a parameter to be fixed later, and let $\lambda_0$ be a positive number such that
		\begin{align}\label{eq: definition of r zero}
		r_0:=\eta\, \lambda_0^{-1/2} < \min\{1/8, h/4\}.
		\end{align}	
We fix arbitrary $\lambda\in\Sigma_\theta$, $|\lambda|>\lambda_0$, put $r:=\eta|\lambda|^{-1/2}<r_0$, and introduce two cut-off functions $\alpha=\alpha_r$, $\beta=\beta_r$, satisfying
		\begin{align*}
		&\alpha \in C^\infty([-h, 0]), \quad 
		\alpha\equiv0 \text{ on } [-h, -h + r ], 
		\quad \alpha\equiv1 \text{ on } [-h+2r, 0], \quad 
		|\partial_z^k\alpha(z)| \le Cr^{-k},
		\\
		&\beta \in C^\infty([-h, 0]), \quad 
		\beta\equiv1 \text{ on } [-h, -h+2r], \quad 
		\beta\equiv0 \text{ on } [-h+3r, 0], \quad 
		|\partial_z^k\beta(z)| \le Cr^{-k}
		\end{align*}
for $k=0,1,2$, compare the proof of Lemma~\ref{LemmaAnisotropicLaplaceResolventOmega}.
We then split the estimate for $\partial_zv$ into the ``upper'' and ``lower'' parts in $\Omega$ as
		\begin{equation} \label{eq: decomposition to alpha and beta}
		\|\partial_z v\|_{L^\infty_H L^p_z(\Omega)} 
		\le \|\partial_z (\alpha v)\|_{L^\infty_H L^p_z(\Omega)} 
		+ \|\partial_z (\beta v)\|_{L^\infty_H L^p_z(\Omega)}.
		\end{equation}
\textbf{Step 1.}
Let us first focus on $\partial_z(\alpha v)$.
By Lemma~\ref{lem: anisotropic interpolation inequality} with radius $\lvert \lambda\rvert^{-1/2}$ and $p=q$ we have
		\begin{align}\label{eq: alpha estimate step one}
		\begin{split}
		|\lambda|^{1/2} \|\partial_z(\alpha v)\|_{L^\infty_H L^p_z(\Omega)} 
		\le C_p |\lambda|^{1/p} \sup_{x_0' \in G} \biggl(
		|\lambda|^{1/2} 
		\|\partial_z (\alpha v)\|_{L^p(C(x_0'; |\lambda|^{-1/2}))} 
		+ \|\nabla_H\partial_z (\alpha v)\|_{L^p(C(x_0'; |\lambda|^{-1/2}))}
		\biggr),
		\end{split}
		\end{align}
where $C(x_0'; |\lambda|^{-1/2})$ denotes the cylinder $B(x_0'; |\lambda|^{-1/2})\times(-h,0)$ and we used that 
		\[
		\lVert f\rVert_{L^\infty_H L^p_z(\Omega)}=\sup_{x_0'\in G}\lVert f\rVert_{L^\infty(B(x'_0,R);L^p_z)}, \quad R>0.
		\]
In the following we fix arbitrary $x_0' \in G$ and introduce a cut-off function $\theta = \theta_r \in C^\infty_c(\mathbb R^2)$ such that
		\begin{equation*}
		\theta\equiv1 \text{ in } \overline{B(x_0'; |\lambda|^{-1/2})}, \quad 	
		\mathrm{supp}\,\theta \subset B(x_0'; r), \quad 
		\|\nabla_H^k \theta\|_{L^\infty(\mathbb R^2)} 
		\le Cr^{-k}
		\end{equation*}
for $k=0,1,2$. Then $\theta\alpha v$ solves
		\begin{align*}
		\lambda(\theta\alpha v) - \Delta(\theta\alpha v) 
		&= - \theta\alpha\nabla_H\pi 
 	 	- 2\nabla(\theta\alpha)\cdot\nabla v 
 	  	- (\Delta(\theta\alpha)) v \quad\text{on}\quad \Omega', 
		\\
		\partial_z(\theta\alpha v)|_{\Gamma_u' \cup \Gamma_b'} &= 0, \quad 
		\theta\alpha v \text{ periodic on }\Gamma_l'.
		\end{align*}
We further differentiate this equation with respect to $z$ to obtain
		\begin{align*}
		\lambda(\theta \partial_z(\alpha v) ) - \Delta(\theta \partial_z(\alpha v)) 
		= F_1+\partial_z F_2 \quad\text{on}\quad \Omega',
		\quad
		\restr{\theta\partial_z(\alpha v)}{\Gamma'_u \cup \Gamma'_b} = 0,\quad 
		\theta\partial_z(\alpha v) \text{ periodic on }\Gamma_l'.
		\end{align*}
where 
		\begin{align*}
		F_1&:=-\theta (\partial_z\alpha) (\nabla_H\pi)
	 	-(\Delta_H\theta)(\partial_z\alpha)v 
		-(\Delta_H\theta)\alpha(\partial_z v),
		\\
		F_2&:=-2(\nabla_H\theta)\alpha\cdot(\nabla_Hv) 
		-2\theta(\partial_z\alpha)(\partial_zv) 
		- \theta(\partial_z^2\alpha)v.
		\end{align*}
By \eqref{eq:AnisotropicResolventOmega} and \eqref{eq:AnisotropicResolventDerivativeOmega} for $\Omega'$ in the case $q=p$, we obtain the estimate
		\begin{align}\label{eq: alpha estimate step two}
		|\lambda|^{1/2} \|\partial_z (\theta \alpha v)\|_{L^p(\Omega')}
		+ \|\nabla\partial_z (\theta \alpha v)\|_{L^p(\Omega')}
		\le C_{\theta} \left(
		|\lambda|^{-1/2}\lVert F_1\rVert_{L^p(\Omega')} 
		+ \lVert F_2\rVert_{L^p(\Omega')}
		\right).
		\end{align}
and since $\theta \equiv 1$ on $C(x'_0;\lvert \lambda \rvert^{-1/2})\subset \Omega'$ by \eqref{eq: definition of r zero}, we further have
		\begin{align}\label{eq: alpha estimate step three}
		\begin{split}
		\|\partial_z (\alpha v)\|_{L^p(C(x_0'; |\lambda|^{-1/2}))}
		&\le \|\partial_z (\theta \alpha v)\|_{L^p(\Omega')},
		\\
		\|\nabla\partial_z (\alpha v)\|_{L^p(C(x_0'; |\lambda|^{-1/2}))}
		&\le \|\nabla\partial_z (\alpha v)\|_{L^p(C(x_0'; |\lambda|^{-1/2}))}.
		\end{split}
		\end{align}
Let us estimate each term on this right-hand side of \eqref{eq: alpha estimate step two} as follows:
Denoting $\lVert \cdot \rVert_{L^p_H}:=\lVert \cdot \rVert_{L^p(B(x'_0;r))}$ and $\lVert \cdot \rVert_{L^p_z}:=\lVert \cdot \rVert_{L^p(-h,0)}$, we first observe that the cut-off functions satisfy
		\[
		\lVert \theta\rVert_{L^p_H}
		\le C r^{2/p},
		\quad 
		\lVert \nabla_H \theta\rVert_{L^p_H}
		\le C r^{2/p-1},
		\quad 
		\lVert \Delta_H \theta\rVert_{L^p_H}
		\le C r^{2/p-2}
		\]
as well as
		\[
		\|\partial_z\alpha\|_{L^p_z}
		\le C r^{1/p-1},
		\quad 		
		\|\partial_z^2\alpha\|_{L^p_z}
		\le C r^{1/p-2}.
		\]
		%
By Proposition~\ref{prop: nearly-optimal estimate for laplacian} we then have
		\begin{equation*}
		\|\theta (\partial_z\alpha) (\nabla_H\pi)\|_{L^p(\Omega')} 
		\le \lVert \theta\rVert_\infty 
		\|\partial_z\alpha\|_{L^p_z}
		\|\nabla_H\pi\|_{L^p_H} 
		\le C_p r^{3/p-1}(1 + |\log r|) \|f\|_{L^\infty_H L^p_z(\Omega)}.
		\end{equation*}		
		%
We further have the Poincar\'e inequality 
		\begin{align}\label{eq: vertical poincare}
		\|f\|_{L^\infty(G'; L^p(-h, -h + d))} 
		\le d\|\partial_z f\|_{L^\infty_H L^p_z},
		\quad 0\le d\le h, \quad \restr{f}{\Gamma'_b} = 0
		\end{align}
and hence using H\"older's inequality yields
		\begin{equation*}
		\|(\Delta_H\theta)(\partial_z\alpha)v\|_{L^p(\Omega')} 
		\le 
		\|\Delta_H\theta\|_{L^p_H} 
		\|\partial_z\alpha\|_{\infty}
		\|v\|_{L^\infty(G'; L^p(-h,-h+2r))} 
		\le Cr^{2/p-2}\|\partial_zv\|_{L^\infty_H L^p_z(\Omega)}.
		\end{equation*}
		%
For the third term in $F_1$ we simply have
		\begin{equation*}
		\|(\Delta_H\theta)\alpha(\partial_zv)\|_{L^p(\Omega')} 
		\le \lVert \Delta_H \theta\rVert_{L^p_H}
		\lVert \alpha\rVert_\infty 
		\|\partial_zv\|_{L^\infty_H L^p_z(\Omega)}
		\le Cr^{2/p-2} \|\partial_zv\|_{L^\infty_H L^p_z(\Omega)}.
		\end{equation*}	
		%
The first term in $F_2$ is estimated via \eqref{eq: horizontal estimate for v two}, yielding
		\begin{equation*}
		\|(\nabla_H\theta) \alpha (\nabla_Hv)\|_{L^p(\Omega')} 
		\le \lVert \nabla_H \theta\rVert_{L^p_H} 
		\lVert \alpha\rVert_\infty
		\lVert \nabla_Hv\|_{L^\infty_H L^p_z(\Omega)} 
		\le C_{\theta,p,\lambda_0} r^{2/p-1} |\lambda|^{-1/2} \|f\|_{L^\infty_H L^p_z(\Omega)},
		\end{equation*}
		%
whereas for the second term in $F_2$ we simply have
		\begin{equation*}
		\|\theta (\partial_z\alpha) (\partial_zv)\|_{L^p(\Omega')} 
		\le \lVert \theta\rVert_\infty \lVert \partial_z \alpha\rVert_\infty
		\lVert \partial_z v\rVert_{L^\infty_H L^p_z(\Omega)}
		\le Cr^{2/p-1} \|\partial_zv\|_{L^\infty_H L^p_z(\Omega)},
		\end{equation*}
		%
and by the Poincar\'e inequality \eqref{eq: vertical poincare} we estimate the last term by
		\begin{equation*}
		\|\theta (\partial_z^2\alpha) v\|_{L^p(\Omega')} 
		\le \lVert \theta\rVert_{L^p_H}\lVert \partial^2_z \alpha\rVert_\infty 
		\lVert  v\rVert_{L^\infty_H L^p_z(\Omega)}
		\le Cr^{2/p-1} \|\partial_zv\|_{L^\infty_H L^p_z(\Omega)}.
		\end{equation*}
		%
Collecting the above estimates, using \eqref{eq: alpha estimate step one}, \eqref{eq: alpha estimate step two} and \eqref{eq: alpha estimate step three}, as well as $r=\eta \lvert \lambda\rvert^{-1/2}$, we obtain that
		\begin{align} \label{eq: estimate for alpha v}
		\begin{split}
		|\lambda|^{1/2} \|\partial_z(\alpha v)\|_{L^\infty_HL^p_z(\Omega)} 
		&\le C_{\theta,p,\lambda_0} \left(
		\eta^{2/p-2}+\eta^{3/p-2}\lvert \lambda \rvert^{-1/2p}
		+\eta^{2/p-1}r^{1/p}\lvert \log(r)\rvert
		\right)
		\lVert f\rVert_{L^\infty_H L^p_z(\Omega)}
		\\&+ C_{\theta,p} (\eta^{2/p-1} + \eta^{2/p-2}) |\lambda|^{1/2} 
		\|\partial_zv\|_{L^\infty_HL^p_z(\Omega)}.
		\\
		&\le C_{\theta,p,\lambda_0} \eta^{2/p-1}
		\big(1 + r^{1/p} |\log r| \big) \|f\|_{L^\infty_HL^p_z(\Omega)} 
		\\&+ C_{\theta,p} (\eta^{2/p-1} + \eta^{2/p-2}) |\lambda|^{1/2} 
		\|\partial_zv\|_{L^\infty_HL^p_z(\Omega)}.
		\end{split}
		\end{align}
\textbf{Step 2:} Now we shall estimate $\partial_z(\beta v)$.
We apply Lemma~\ref{lem: anisotropic interpolation inequality} as in the previous step to obtain
		\begin{align}\label{eq: beta estimate step one}
		|\lambda|^{1/2} \|\partial_z(\beta v)\|_{L^\infty_H L^p_z(\Omega)} 
		&\le C_p |\lambda|^{1/p} \sup_{x_0' \in G} \left(
		|\lambda|^{1/2} \|\partial_z (\beta v)\|_{L^p(C(x_0'; |\lambda|^{-1/2}))} 
		+ \|\nabla_H\partial_z (\beta v)\|_{L^p(C(x_0'; |\lambda|^{-1/2}))}
		\right).
		\end{align}
In the following we fix an arbitrary point $x_0' \in G$.
With the same cut-off function $\theta \in C^\infty_c(\mathbb R^2)$ as in Step 1, we find that $\theta\beta v$ solves
		\begin{align*}
		\lambda(\theta\beta v) - \Delta(\theta\beta v) 
		= F_3 \quad\text{in}\quad \Omega', 
		\quad
		\partial_z(\theta\beta v)|_{\Gamma_u'} = 0, \quad 
		\restr{\theta\beta v}{\Gamma_b'} = 0, \quad 
		\theta\beta v \text{ periodic on }\Gamma_l'
		\end{align*}
where
		\[
		F_3:=
		-\theta\beta(\nabla_H\pi) 
		-2(\nabla_H \theta)\beta\cdot (\nabla_H v)		
		-2\theta(\partial_z \beta)(\partial_z v) 
		-(\Delta_H \theta)\beta v
		-2\theta(\partial_z^2\beta)v.
		\]	
We apply estimate \eqref{eq:AnisotropicResolventOmega} on $\Omega'$ with $q=p$ to obtain
		\begin{align}\label{eq: beta estimate step two}
		|\lambda|^{1/2} \|\nabla (\theta \beta v)\|_{L^p(\Omega')} 
		+ \|\Delta (\theta \beta v)\|_{L^p(\Omega')} 
		\le C_\theta \lVert F_3\rVert_{L^p(\Omega')}
		\end{align}	
where we further have, compare \eqref{eq: alpha estimate step three}, that
		\begin{align}\label{eq: beta estimate step three}
		\begin{split}
		\lVert \partial_z (\beta v)\rVert_{L^p(C(x'_0;\lvert \lambda \rvert^{-1/2})}
		&\le \lVert \nabla (\theta\beta v)\rVert_{L^p(\Omega')},
		\\
		\lVert \nabla_H \partial_z (\beta v)\rVert_{L^p(C(x'_0;\lvert \lambda \rvert^{-1/2})}
		&\le \lVert \nabla_H \partial_z (\theta \beta v)\rVert_{L^p(\Omega')}
		\le \lVert \theta \beta v \rVert_{W^{2,p}(\Omega')}
		\le C_p \lVert \Delta (\theta \beta v)\rVert_{L^p(\Omega')} 
		\end{split}
		\end{align}
by the invertibility of the Laplace operator with mixed Neumann and Dirichlet boundary conditions, compare Section \ref{sec:laplace}.

We now estimate the right-hand side of \eqref{eq: beta estimate step two} as follows:
Note that $\beta$ satisfies the estimates
		\[
		\lVert \beta\rVert_{L^p_z}\le C r^{1/p},
		\quad 
		\lVert \partial_z \beta\rVert_{L^p_z}\le C r^{1/p-1},
		\quad
		\lVert \partial_z^2 \beta\rVert_{L^p_z}\le C r^{1/p-2}
		\]
since $\supp (\beta)\subset [-h,-h+3r]$. 
It follows from Proposition~\ref{prop: nearly-optimal estimate for laplacian} 				that
		\begin{align*}
		\|\theta \beta (\nabla_H\pi)\|_{L^p(\Omega')} 
		\le \lVert \theta\rVert_\infty
		\|\beta\|_{L^p_z}
		\|\nabla_H\pi\|_{L^p_H} 
		\le C_p r^{3/p} (1 + |\log r|) \|f\|_{L^\infty_HL^p_z(\Omega)}.
		\end{align*}
		%
The estimate \eqref{eq: horizontal estimate for v two} implies that
		\begin{align*}
		\|(\nabla_H\theta)\beta\cdot (\nabla_H v)\|_{L^p_H} 
		&\le \|\nabla_H\theta\|_{L^p_H} 
		\lVert \beta\rVert_\infty		
		\|\nabla_H v\|_{L^\infty_HL^p_z(\Omega)} 
		\le C_{\theta,p,\lambda_0}r^{2/p-1} |\lambda|^{-1/2} 
		\|f\|_{L^\infty_HL^p_z(\Omega)},
		\end{align*}
		%
and for the term containing vertical derivatives we have
		\begin{equation*}
		\|\theta(\partial_z\beta) (\partial_zv)\|_{L^p(\Omega')} 
		\le \|\theta\|_{L^p_H}
		\|\partial_z\beta\|_\infty 
		\|\partial_zv\|_{L^\infty_HL^p_z(\Omega)} 
		\le Cr^{2/p-1} \|\partial_zv\|_{L^\infty_HL^p_z(\Omega)}.
		\end{equation*}
		%
By the Poincar\'e inequality \eqref{eq: vertical poincare} we have
		\begin{align*}
		\|(\Delta_H\theta) \beta v\|_{L^p(\Omega')} 
		\le \|\Delta_H\theta\|_{L^p_H}
		\lVert \beta\rVert_\infty		
		\|v\|_{L^\infty(G; L^p(-h, -h + 3r))} 
		\le Cr^{2/p-1} \|\partial_zv\|_{L^\infty_HL^p_z(\Omega)}
		\end{align*}
		%
as well as
		\begin{equation*}
		\|\theta(\partial_z^2\beta) \, v\|_{L^p(\Omega')} 
		\le \lVert \theta\rVert_{L^p_H}
		\lVert \partial_z^2\beta \rVert_\infty
		\lVert v\rVert_{L^\infty(G;L^p(-h,-h+3r))}
		\le Cr^{2/p-1} \|\partial_zv\|_{L^\infty_HL^p_z(\Omega)}.
		\end{equation*}
		%
		%
Combining the above estimates with \eqref{eq: beta estimate step one}, \eqref{eq: beta estimate step two} and \eqref{eq: beta estimate step three} as well as $r=\eta \lvert \lambda \rvert^{-1/2}$ then yields
		\begin{align} \label{eq: estimate for beta v}	
		\begin{split}
		|\lambda|^{1/2} \|\partial_z(\beta v)\|_{L^\infty_H L^p_z(\Omega)} 
		&\le C_{\theta,p,\lambda_0}
		\left(
		\eta^{2/p-1}
		+\eta^{3/p} |\lambda|^{-1/2p} 
		\big( 1 + |\log(\eta|\lambda|^{-1/2})| \big) 
		\right)\|f\|_{L^\infty_H L^p_z(\Omega)} 
		\\&+ C_{\theta,p}\eta^{2/p-1} |\lambda|^{1/2} \|\partial_zv\|_{L^\infty_H L^p_z(\Omega)}.
		\end{split}
		\end{align}
We now substitute \eqref{eq: estimate for alpha v} and \eqref{eq: estimate for beta v} into \eqref{eq: decomposition to alpha and beta}. Since all constants $C>0$ do not depend on the parameter $\eta>0$, we can take it to be sufficiently large and so similarly to the proof of Lemma ~\ref{LemmaAnisotropicLaplaceResolventOmega} we obtain
		\begin{equation*}
		|\lambda|^{1/2} \|\partial_z v\|_{L^\infty_HL^p_z(\Omega)} 
		\le C_{\theta,p,\lambda_0}
		\left(
		\eta^{2/p-1}(1+r^{1/p}\lvert \log(r)\rvert)
		+\eta^{3/p}\lvert \lambda \rvert^{-1/2p}(1+\lvert \log(\lvert \lambda \rvert)\rvert)
		\right)
		\|f\|_{L^\infty_HL^p_z(\Omega)} .
		\end{equation*}
Since 
		\[
		\sup_{0<r<r_0}r^{1/p}\lvert \log(r)\rvert <\infty,
		\quad 
		\sup_{\lvert \lambda\rvert>\lambda_0}
		\lvert \lambda\rvert^{-1/2p}(1+\lvert \log(\lvert \lambda\rvert)\rvert)		
		<\infty,
		\]
for any $r_0,\lambda_0>0$ and $p\in (1,\infty)$, this implies the desired estimate 
$|\lambda|^{1/2} \|\partial_z v\|_{L^\infty_HL^p_z(\Omega)} 
\le C \|f\|_{L^\infty_HL^p_z(\Omega)}$ for $|\lambda| \ge \lambda_0$.
\end{proof}

We now turn to the problem
		\begin{align}\label{eq:HydrostaticStokesResolventDerivative}
		\lambda v-Av
		=\PP\partial_z f 
		\text{ on } \Omega
		\end{align}
with boundary conditions \eqref{eq:bc} for $f\in \Xip$. Since 
		\begin{align}\label{eq:FixedByProjection}
		\PP\partial_z f=\partial_z f-(1-Q)\overline{\partial_z f}=\partial_z f, 
		\end{align}
whenever $f=0$ on $\Gamma_u\cup \Gamma_b$ and $C^\infty_{\text{per}}([0,1]^2;C^\infty_c (-h,0))^2$ is dense in $\Xip$ we may assume without loss of generality that \eqref{eq:FixedByProjection} holds.
Moreover, in view of periodic extension we may assume that \eqref{eq:HydrostaticStokesResolventDerivative} holds in a larger domain $\Omega' := G'\times(-h,0)$, $G' := (-2,3)^2$. Since the problem is well-posed in $\lpso$ by \eqref{eq:LpEstimateDerivatives}, estimate~\eqref{eq:SecondVerticalEstimate} then follows from the following:

\begin{proposition} \label{thm: 2nd vertical derivative estimate}
Let $p\in(2, \infty)$ and $\theta\in(0, \pi)$.
Then there exists constants $\lambda_0>0$ and $C_{\theta,p,\lambda_0}>0$ such that for all $\lambda\in\Sigma_\theta$ with $\lvert \lambda \rvert>\lambda_0$ and $f \in \Xip$ the solution to the problem \eqref{eq:HydrostaticStokesResolventDerivative} satisfies
		\begin{equation*}
		\lvert \lambda\rvert^{1/2}\lVert  v\rVert_{L^\infty_HL^p_z(\Omega)}
		\le C_{\theta, p,\lambda_0} \lVert f\rVert_{L^\infty_HL^p_z(\Omega)}.
		\end{equation*}
\end{proposition}
To prove this estimate, 
we adopt a duality argument combined with
the use of a regularized delta function, which is based on the methodology known in
$L^\infty$-type error analysis of the finite element method, cf. \cite{RaSc82}.

In order to prove this estimate we first introduce some notation. Using periodicity, one sees that for any $\varepsilon\in (0,1)$ we have $B(x_0',\varepsilon)\subset G'$ for $x_0'\in G$ and 
		\begin{equation*}
		\|v\|_{L^\infty_HL^p_z(\Omega)}^p 
		= \sup_{x_0' \in G} \sup_{x' \in B(x'_0; \varepsilon)} 
		\int_{-h}^0 |v(x', z)|^p \, dz,
		\end{equation*}
where by $B(x'_0;\varepsilon)$ we continue to denote a disk in $\R^2$, compare Lemma~\ref{lem: anisotropic interpolation inequality}.
In the following we fix arbitrary $x_0' \in G$, $x' \in B(x_0'; \varepsilon)$ and choose $\varepsilon = |\lambda|^{-\frac{p}{2(p-2)}}$ for $\lambda$ as above.

Letting $\delta\ge 0$ be a smooth nonnegative function in the variables $(x,y)=:x'$ such that $\mathrm{supp}\,\delta \subset B(0; 1)$ and $\int_{\mathbb R^2} \delta \, dx' = 1$, we introduce a rescaled function as
		\begin{align}\label{eq:DeltaRescaling}
		\delta_\varepsilon(x') 
		:= \frac1{\varepsilon^2} \delta \left( \frac{x'}\varepsilon \right ), 
		\qquad 
		\delta_{\varepsilon, x_0'}(x') 
		:= \delta_\varepsilon(x' - x_0').
		\end{align}
We then obtain 
		\begin{equation} \label{eq1: proof of 2nd vertical estimate}
		\int_{-h}^0 |v(x', z)|^p \, dz 
		= \int_{-h}^0\int_{G'}
		\big( |v(x', z)|^p - |v(y', z)|^p \big) \delta_{\varepsilon,x'_0}(y') 
		\, dy'dz
		+ (v, \delta_{\varepsilon, x_0'}|v|^{p-2}v^*)_{\Omega'} 
		=: I_1(x') + I_2,
		\end{equation}
where $v^*$ means the complex conjugate of $v$ and $(\cdot, \cdot)_{\Omega'}$ denotes the inner product on $L^2(\Omega')^2$. In the following we estimate the two terms on the right-hand side separately, beginning with $I_1$.

\begin{lemma} \label{lem1: proof of 2nd vertical estimate}
Under the assumptions of Proposition~\ref{thm: 2nd vertical derivative estimate} we have for all 
for all $x'_0\in G$ and $x'\in B(x'_0;\varepsilon)$, $\varepsilon = |\lambda|^{-\frac{p}{2(p-2)}}$, that
		\[
		|I_1(x')|=\left\vert \int_{\Omega'} 
		\big( |v(x', z)|^p - |v(y', z)|^p \big) \delta_{\varepsilon,x'_0}(y') 
		\, dy'dz \right\vert 
		\le C_{\theta,p}|\lambda|^{-1/2} \|f\|_{L^\infty_H L^p_z(\Omega)} \|v\|_{L^\infty_H L^p_z(\Omega)}^{p-1}.
		\]
\end{lemma}

\begin{proof}
Since $\int_{\R^2}\delta_{\varepsilon,x'_0}(y')\,dy'=1$ and 
$\mathrm{supp}\, \delta_{\varepsilon,x_0'} \subset B(x_0'; \varepsilon)$ we obtain
		\begin{align*}
		|I_1(x')| 
		&\le \sup_{y' \in B(x_0'; \varepsilon)} \int_{-h}^0 
		\big| |v(x', z)|^p - |v(y', z)|^p \big| \, dz 
		\\
		&\le C \sup_{y' \in B(x_0'; \varepsilon)} \int_{-h}^0 
		(|v(x', z)|^{p-1} + |v(y', z)|^{p-1}) \big| v(x', z) - v(y', z) \big| 
		\, dz,
		\end{align*}
where we have used the elementary inequality
		\begin{equation*}
		|a^p - b^p| 
		\le p \max\{a, b\}^{p-1} |a - b| 
		\le p (a+b)^{p-1} |a - b| 
		\le p 2^{p-2} (a^{p-1}+b^{p-1}) |a - b|
		\end{equation*}
for all $a,b\ge 0$, where we used that $p\in [2,\infty)$ implies that $x\mapsto x^{p-1}$ is a convex function.
H\"older's inequality then implies that
		\begin{equation*}
		\int_{-h}^0 
		(|v(x', z)|^{p-1} + |v(y', z)|^{p-1}) \big| v(x', z) - v(y', z) \big| 
		\, dz 
		\le (\|v(x')\|_{L^p_z}^{p-1} + \|v(y')\|_{L^p_z}^{p-1}) 
		\|v(x') - v(y')\|_{L^p_z}.
		\end{equation*}
		%
		%
Hence we have
		\begin{align*}
		\sup_{x'\in B(x'_0;\varepsilon)}|I_1(x')| 
		&\le C \sup_{y'\in B(x_0'; \varepsilon)} \|v(y')\|_{L^p_z}^{p-1} 
		\sup_{y'\in B(x_0'; \varepsilon)} \|v(x') - v(y')\|_{L^p_z} 
		\le C \|v\|_{L^\infty_HL^p_z}^{p-1}	
		\varepsilon^\alpha 
		\|v\|_{C^\alpha_HL^p_z(\Omega)},
		\end{align*}
where $\alpha := 1-2/p>0$ and 
$\lVert v \rVert_{C^\alpha_HL^p_z(\Omega)}$ denotes the space of $L^p(-h,0)$-valued H\"older continuous functions of exponent $\alpha$ on $\overline{G}$.

The assumption $\varepsilon = |\lambda|^{-\frac{p}{2(p-2)}}$ then yields
$\varepsilon^\alpha = |\lambda|^{-1/2}$.
We now use the Sobolev embedding $W^{1,p}(G) \hookrightarrow C^\alpha(\overline{G})$ to obtain the estimate $\|v\|_{C^\alpha_HL^p_z} \le C\|v\|_{W^{1,p}(\Omega)}$.
In addition, the Poincar\'e inequality yields 
		\[
		\lVert v \rVert_{W^{1,p}(\Omega)}
		\le C_p \lVert \nabla v\rVert_{L^p(\Omega)}
		= C_p\lVert \nabla (\lambda-\Apos)^{-1}\partial_z f\rVert_{L^p(\Omega)}
		\le C_{\theta,p} \lVert f\rVert_{L^p(\Omega)}
		\le C_{\theta,p} \lVert f\rVert_{L^\infty_H L^p_z(\Omega)},
		\]
where we used that $\nabla (-\Apos)^{-1/2}$, $\Apos(\lambda-\Apos)^{-1}$ and $(-\Apos)^{-1/2}\partial_z$ are (uniformly) bounded on $\lpso$ for $\lambda\in\Sigma_\theta$ 
by \cite{GGHHK17}. 
%
Combining these results then gives the desired estimate.
\end{proof}
In order to estimate $I_2$ we perform a duality argument. 
For this purpose we introduce an auxiliary problem corresponding to \eqref{eq:HydrostaticStokesResolventDerivative} as follows:
		\begin{equation} \label{eq: dual problem}
		\begin{aligned}
		\lambda^* w - \Delta w + \nabla_H \Pi 
		&= \delta_{\varepsilon,x_0'} |v|^{p-2}v^* \quad\text{in}\quad \Omega', 
		\\
		\partial_z \Pi &= 0 \quad\text{in}\quad \Omega', 
		\\
		\mathrm{div}_H\, \bar w &= 0 \quad\text{in}\quad G', 
		\\
		\partial_zw|_{\Gamma_u'} = 0, \quad w|_{\Gamma_b'} &= 0, \quad 
		w, \Pi \text{ periodic on } \Gamma_l',
		\end{aligned}
		\end{equation}
where the upper script $*$ means complex conjugate as before.
We establish an $L^1_HL^q_z$-estimate to this problem, where $q := p/(p-1)$ is the dual index of $p$.
\begin{proposition} \label{prop: L1Lq estimate for dual problem}
Let $p\in(2,\infty)$, $1/p+1/q=1$ and $\theta\in(0,\pi)$.
Then there exists a sufficiently large 
$\lambda_0>0$ and a constant $C_{p, \lambda_0, \theta}>0$ such that the solution of \eqref{eq: dual problem} satisfies
		\begin{equation*}
		|\lambda|^{1/2}\|\partial_zw\|_{L^1_HL^q_z(\Omega')} 
		\le C_{\theta,p} \left(
		1+|\lambda|^{-1/2q}\varepsilon^{2/s-2}
		\right)
		\|v\|_{L^\infty_H L^p_z(\Omega)}^{p-1},
		\end{equation*}
for all $\varepsilon\in(0,1)$, $s \in (1, q]$, $x_0' \in G$, $\lambda\in\Sigma_\theta$, $|\lambda|>\lambda_0$, and $v \in \Xip$.
\end{proposition}
\begin{remark}
	If one even has $p\in (3,\infty)$ then this result can be extended to the full range of $\lambda \in \Sigma_\theta$ by a similar argument as in the proof of Lemma~\ref{LemmaAnisotropicLaplaceResolventOmega}, compare Remark~\ref{FullRangeRemark}.
\end{remark}

%
%
%
%
%

For simplicity, we write $L^p_HL^q_z$ to refer to $L^p_H L^q_z(\Omega')=L^p(G'; L^q(-h, 0))$ when there is no ambiguity.
First we introduce the following result.

\begin{lemma} \label{lem: Ls estimate of F}
Let $\varepsilon\in(0,1)$, $x_0' \in G$, $p\in(1,\infty)$, $1/p+1/q=1$ and $v \in \Xip$ be arbitrary. Then, for $\delta_{\varepsilon,x'_0}$ defined as in \eqref{eq:DeltaRescaling} and
$s \in [1,q]$ we have
		\begin{equation*}
		\|\delta_{\varepsilon,x_0'} |v|^{p-2}v^*\|_{L^s(\Omega')} 
		\le C\|\delta_{\varepsilon,x_0'} |v|^{p-2}v^*\|_{L^s_HL^q_z} 
		\le C\varepsilon^{2/s-2} \|v\|_{L^\infty_HL^p_z(\Omega)}^{p-1}
		\end{equation*}
for a constant $C>0$ not depending on $\varepsilon$, $x'_0$ and $v$.
\end{lemma}
%
%
\begin{proof}
We set $F := \delta_{\varepsilon,x_0'} |v|^{p-2}v^*$. 
Noting that $|F|^q = \delta_{\varepsilon, x_0'}^q |v|^p$ and that $\delta_{\varepsilon,x_0'}$ is independent of $z$, we obtain
		\begin{align*}
		\|F\|_{L^s_HL^q_z} 
		&= \left[ 
		\int_{G'} \left( 
		\int_{-h}^0 \delta_\varepsilon(x' - x_0')^q |v(x', z)|^p\,dz 
		\right)^{s/q}\,dx' 
		\right]^{1/s} 
		\\
		&\le \left( 
		\int_{G'} \delta_\varepsilon(x' - x_0')^s\,dx' 
		\right)^{1/s} 
		\left[ 
		\sup_{x'\in G'} \left( \int_{-h}^0 |v(x', z)|^p\, dz \right)^{1/p} 
		\right]^{p/q} 
		\\
		&\le C\varepsilon^{2/s - 2} \|v\|_{L^\infty_HL^p_z(\Omega)}^{p-1},
		\end{align*}
where we used the periodicity of $v$ in the last step. This completes the proof.
\end{proof}
\begin{proof}[Proof of Proposition~\ref{prop: L1Lq estimate for dual problem}]
We set $r := \eta|\lambda|^{-1/2}$, where $\eta>0$ is a large number to be fixed later and $\lvert \lambda\rvert>\lambda_0$, where $\lambda_0>0$ is sufficiently large such that $\eta \lambda_0^{-1/2}<1$.
We introduce two cut-off functions $\alpha=\alpha_r$, $\beta=\beta_r$ in the vertical direction as follows:
		\begin{align*}
		&\alpha \in C^\infty([-h, 0]), \quad 
		\alpha\equiv0 \text{ in } [-h, -h + r ], \quad 
		\alpha\equiv1 \text{ in } [-h+2r, 0], \quad 
		|\partial_z^k\alpha(z)| \le Cr^{-k}, 
		\\
		&\beta \in C^\infty([-h, 0]), \quad 
		\beta\equiv1 \text{ in } [-h, -h+2r], \quad 
		\beta\equiv0 \text{ in } [-h+3r, 0], \quad 
		|\partial_z^k\beta(z)| \le Cr^{-k}
		\end{align*}
for $k=0,1,2$.
Then we may split the estimate for $\partial_zw$ into the ``upper'' and ``lower'' parts in $\Omega'$ as
		\begin{equation} \label{eq: decomposition to alpha and beta for w}
		\|\partial_z w\|_{L^1_HL^q_z} 
		\le \|\partial_z (\alpha w)\|_{L^1_HL^q_z} 
		+ \|\partial_z (\beta w)\|_{L^1_HL^q_z}.
		\end{equation}
\textbf{Step 1.} We consider $\alpha w$, which satisfies
		\begin{align*}
		\lambda^* \alpha w - \Delta(\alpha w) 
		&= \alpha F 
		- \alpha(\nabla_H \Pi) 
		- 2(\partial_z\alpha) (\partial_z w) 
		- (\partial_z^2\alpha) w,
		\\
		\partial_z(\alpha w) &= 0 \quad\text{on}\quad \Gamma_u'\cup\Gamma_b', 
		\qquad 
		\alpha w \text{ periodic on }\Gamma_l'
		\end{align*}
where $F:= \delta_{\varepsilon,x_0'} |v|^{p-2}v^*$ as in the proof of Lemma\ref{lem: Ls estimate of F}.
Differentiating this with respect to $z$ yields
		\begin{align*}
		\lambda^*\partial_z(\alpha w) - \Delta(\partial_z(\alpha w)) 
		&= \partial_z\left[
		\alpha F
		- 2(\partial_z\alpha) (\partial_z w)
		- (\partial_z^2\alpha) w
		\right]
		- (\partial_z\alpha) (\nabla_H \Pi) \quad\text{in}\quad \Omega', 
		\\
		\partial_z(\alpha w) &= 0 \quad\text{on}\quad \Gamma_u'\cup\Gamma_b', 
		\qquad 
		\partial_z(\alpha w) \text{ periodic on }\Gamma_l'.
		\end{align*}
Applying Lemma~\ref{LemmaAnisotropicLaplaceResolventOmega} in $L^1_HL^q_z(\Omega')$ we obtain
		\begin{align*}
		|\lambda|^{1/2} \|\partial_z(\alpha w)\|_{L^1_HL^q_z} 
		\le &C\left(
		\|\alpha F\|_{L^1_HL^q_z} 
		+ \|(\partial_z\alpha) (\partial_zw)\|_{L^1_HL^q_z} 
		+ \|(\partial_z^2\alpha) w\|_{L^1_HL^q_z}
		\right) 
		\\+ &C|\lambda|^{-1/2} \|(\partial_z\alpha) (\nabla_H \Pi)\|_{L^1_HL^q_z}.
		\end{align*}
We now estimate each term on the right-hand side.
By Lemma~\ref{lem: Ls estimate of F} with $s = 1$ we have
		\begin{equation*}
		\|\alpha F\|_{L^1_HL^q_z} 
		\le \lVert \alpha\rVert_\infty \|F\|_{L^1_HL^q_z} 
		\le C \|v\|_{L^\infty_HL^p_z(\Omega)}^{p-1}.
		\end{equation*}
		%
Using the estimate on derivatives of $\alpha$ we obtain
		\begin{equation*}
		\|(\partial_z\alpha) (\partial_zw)\|_{L^1_HL^q_z} 
		\le Cr^{-1} \|\partial_zw\|_{L^1_HL^q_z},
		\end{equation*}
		%
and by the Poincar\'e inequality we have
		\begin{equation*}
		\|(\partial_z^2\alpha) w\|_{L^1_HL^q_z} 
		\le C r^{-2} \|w\|_{L^1_HL^q_z(G'\times (-h,-h+2r))}
		\le Cr^{-1} \|\partial_zw\|_{L^1_HL^q_z}.
		\end{equation*}
		%
Using $L^s(G')\hookrightarrow L^1(G')$ as well as the estimate 
on the pressure term, cf. \cite[Theorem 3.1.]{HieberKashiwabara2015},
in $L^s(\Omega)$ for $s\in (1,q]$, we obtain
		\begin{align*}
		\|(\partial_z\alpha) (\nabla_H \Pi)\|_{L^1_HL^q_z} 
		\le C\|\partial_z\alpha\|_{L^q_z} \|\nabla_H\Pi\|_{L^s(G')}
		\le Cr^{1/q - 1} \|F\|_{L^s} 
		\le Cr^{1/q - 1} \varepsilon^{2/s-2} 
		\|v\|_{L^\infty_HL^p_z(\Omega)}^{p-1},
		\end{align*}
		%
		%
		%
Collecting the above estimates and plugging in $r=\eta\lvert \lambda\rvert^{-1/2}$ yields
		\begin{equation} \label{eq: alpha w}
		|\lambda|^{1/2} \|\partial_z(\alpha w)\|_{L^1_HL^q_z} 
		\le C(1 + \eta^{1/q-1} |\lambda|^{-1/2q} \varepsilon^{2/s-2}) 
		\|v\|_{L^\infty_HL^p_z(\Omega)}^{p-1} 
		+ C	\eta^{-1} |\lambda|^{1/2}
		\|\partial_z w\|_{L^1_HL^q_z}.
		\end{equation}
\textbf{Step 2.} We consider $\beta w$, which satisfies
		\begin{align*}
		\lambda^*\beta w - \Delta(\beta w) 
		&= \beta F 
		- \beta \nabla_H\Pi
		- 2(\partial_z\beta) (\partial_z w) 
		- (\partial_z^2\beta) w 
		\quad\text{in}\quad \Omega', 
		\\
		\partial_z(\beta w) &= 0 \text{ on } \Gamma_u', \quad 
		\beta w = 0 \text{ on } \Gamma_b', 
		\quad \partial_z(\beta w) \text{ periodic on }\Gamma_l'.
		\end{align*}
Applying Lemma~\ref{LemmaAnisotropicLaplaceResolventOmega} in $L^1_HL^q_z$ we obtain
		\begin{equation*}
		|\lambda|^{1/2} \|\partial_z(\beta w)\|_{L^1_HL^q_z} 
		\le C(\|\beta F\|_{L^1_HL^q_z} 
		+ \|(\partial_z\beta) (\partial_zw)\|_{L^1_HL^q_z} 
		+ \|(\partial_z^2\beta) w\|_{L^1_HL^q_z} 
		+ \|\beta (\nabla_H\Pi)\|_{L^1_HL^q_z}).
		\end{equation*}
A calculation similar to Step 1 then gives
		\begin{equation}  \label{eq: beta w}
		|\lambda|^{1/2} \|\partial_z(\beta w)\|_{L^1_HL^q_z} 
		\le C(1 + \eta^{1/q} |\lambda|^{-1/2q} \varepsilon^{2/s-2}) \|v\|_{L^\infty_HL^p_z}^{p-1} 
		+ C\eta^{-1} |\lambda|^{1/2} \|\partial_z w\|_{L^1_HL^q_z}.
		\end{equation}
Substituting \eqref{eq: alpha w} and \eqref{eq: beta w} into \eqref{eq: decomposition to alpha and beta for w} and choosing sufficiently large $\eta$ enable us to absorb the term $\lvert \lambda\rvert^{1/2}\lVert \partial_z w\rVert_{L^1_H L^q_z}$ from the right-hand side, which leads to
		\begin{equation*}
		|\lambda|^{1/2} \|\partial_zw\|_{L^1_HL^q_z} 
		\le C(1 +  |\lambda|^{-1/2q} \varepsilon^{2/s-2}) 
		\|v\|_{L^\infty_HL^p_z(\Omega)}^{p-1}
		\end{equation*}
This completes the proof.
\end{proof}
%
%
%
%
With the preparations above, we are now in the position to prove Proposition~\ref{thm: 2nd vertical derivative estimate}.

\begin{proof}[Proof of Proposition~\ref{thm: 2nd vertical derivative estimate}]
By \eqref{eq1: proof of 2nd vertical estimate} and Lemma~\ref{lem1: proof of 2nd vertical estimate} we have
		\begin{equation} \label{eq1: proof of 2nd vertical derivative estimate}
		\|v\|_{L^\infty_H L^p_z(\Omega)}^p 
		\le C|\lambda|^{-1/2} \|f\|_{L^\infty_H L^p_z(\Omega)} \|v\|_{L^\infty_H L^p_z(\Omega)}^{p-1} 
		+ I_2,
		\end{equation}
with $I_2$ as defined in \eqref{eq1: proof of 2nd vertical estimate}. 		
Substituting \eqref{eq: dual problem} and integrating by parts, we find that
		\begin{align*}
		I_2 
		&= (v, \delta_{\varepsilon, x_0'} |v|^{p-2}v^*)_{\Omega'} 
		= (v, \lambda^* w - \Delta w + \nabla_H \Pi)_{\Omega'} 
		= (\lambda v - \Delta v + \nabla_H\pi, w)_{\Omega'} 
		= (\partial_z f, w)_{\Omega'} \\
		&= -(f, \partial_z w)_{\Omega'},
		\end{align*}
where we have used that
$(v,\nabla_H \Pi)_{\Omega'}=0=(\nabla_H \pi,w)_{\Omega'}$ since $\text{div}_H \overline{v}=0=\text{div}_H \overline{w}$ for the third and
 $f|_{\Gamma_u\cup\Gamma_b} = 0$ for the last equality.
Using $1/p+1/q=1$ and applying Proposition~\ref{prop: L1Lq estimate for dual problem} we obtain
		\begin{align*}
		|I_2| 
		\le \|f\|_{L^\infty_H L^p_z(\Omega)} 
		\|\partial_z w\|_{L^1_H L^q_z} 
		\le C |\lambda|^{-1/2} \|f\|_{L^\infty_H L^p_z(\Omega)} 
		\|v\|_{L^\infty_H L^p_z(\Omega)}^{p-1} 
		\left(
		1+\lvert \lambda\rvert^{-1/2q} \varepsilon^{2/s-2} 
		\right)
		.
		\end{align*}
We set $\varepsilon = |\lambda|^{ -\frac{p}{2(p-2)} }$ for $\lvert \lambda\rvert>1$ and 
$s = \min\{\frac{4p}{3p+2}, \frac{p}{p-1}\} \in (1, q]$. This yields 
		\[
		-\frac1{2q} + \left(1 - \frac1s\right)\frac{p}{p-2} 
		=
		-\frac12 +\frac1{2p}
		+\left(1 - \frac1s\right)\frac{p}{p-2} 
		\le 
		-\frac14 + \frac1{2p} 
		< 0
		\]
which implies that $1+\lvert \lambda\rvert^{-1/2q} \varepsilon^{2/s-2}\le 2$ for $\lvert \lambda\rvert>1$ and therefore
		\begin{equation} \label{eq2: proof of 2nd vertical derivative estimate}
		|I_2| \le C|\lambda|^{-1/2} \|f\|_{L^\infty_H L^p_z(\Omega)} \|v\|_{L^\infty_H L^p_z(\Omega)}^{p-1},
		\quad \lvert \lambda\rvert>1.
		\end{equation}
The desired estimate then follows from \eqref{eq1: proof of 2nd vertical derivative estimate} and \eqref{eq2: proof of 2nd vertical derivative estimate} after dividing by $\|v\|_{L^\infty_H L^p_z(\Omega)}^{p-1}$. 
\end{proof}

\begin{proof}[Proof of Claim~\ref{SemigroupEstimatesHard}]
Estimate \eqref{eq:FirstVerticalEstimate} now follows from \eqref{eq:voneestimate} and Proposition~\ref{vtwoestimate}, whereas estimate \eqref{eq:SecondVerticalEstimate} follows from
Proposition \ref{thm: 2nd vertical derivative estimate}. Estimate 	\ref{eq:ThirdVerticalEstimate} follows from \eqref{eq:FirstVerticalEstimate}, \eqref{eq:SecondVerticalEstimate} and Claim~\ref{SemigroupEstimates}.
\end{proof}

\section{Proof of the main results}\label{sec:proofs}
Theorem~\ref{thm:semigroup} is a direct consequence of Claims~\ref{SemigroupEstimates} and~\ref{SemigroupEstimatesHard}.

For the non-linear problem in the space $\Xip$ we will make use of the following estimates. 

\begin{lemma}\label{NonlinearEstimates}
Let $p>3$. Then exists a constant $C>0$ such that for all $t>0$ and $v_i\in \Xipos$ satisfying $\nabla v_i\in \Xip$ and $\restr{v_i}{\Gamma_b}=0$ with $u_i=(v_i,w_i)$ as in \eqref{eq:WRelation} for $i=1,2$ we have
		\begin{align}
		\tag{i}
		\lVert e^{tA}\mathbb{P}(u_1\cdot \nabla)v_2\rVert_{L^\infty_H L^p_z}
		&\le C t^{-1/2}\lVert \nabla v_1\rVert_{L^\infty_H L^p_z} \lVert v_2\rVert_{L^\infty_H L^p_z},
		\\
		\tag{ii}
		\lVert \nabla e^{tA}\mathbb{P}(u_1\cdot \nabla)v_2\rVert_{L^\infty_H L^p_z}
		&\le C t^{-1/2} \lVert \nabla v_1\rVert_{L^\infty_H L^p_z} \lVert \nabla v_2\rVert_{L^\infty_H L^p_z},
		\\
		\tag{iii}
		\lVert \nabla e^{tA}\mathbb{P}(u_1\cdot \nabla)v_2\rVert_{L^\infty_H L^p_z}
		&\le C t^{-1} \lVert \nabla v_1\rVert_{L^\infty_H L^p_z} \lVert v_2\rVert_{L^\infty_H L^p_z},
		\end{align}
as well as 
		\begin{align}
		\tag{iv}
		\lVert e^{tA}\PP (u_1\cdot \nabla)v_2\rVert_{L^\infty_H L^p_z}
		&\le C \left(
		t^{-1/2}\lVert \nabla v_i\rVert_{L^\infty_H L^p_z} \lVert v_j \rVert_{L^\infty_H L^p_z}
		+ \lVert \nabla v_1\rVert_{L^\infty_H L^p_z} \lVert \nabla v_2\rVert_{L^\infty_H L^p_z}
		\right)
		\end{align}
where $\{i,j\}=\{1,2\}$.
\end{lemma}

\begin{proof}
We begin by noting that
		\begin{align*}
		\lVert (u_1 \cdot\nabla) v_2\rVert_{L^\infty_H L^p_z}
		&\le  \left(
		\lVert v_1 \rVert_{L^\infty(\Omega)}+\lVert w_1\rVert_{L^\infty(\Omega)}
		\right)
		\lVert \nabla v_2\rVert_{L^\infty_H L^p_z}.
		\end{align*}
So, using Sobolev embeddings, the Poincar{\'e} inequality and $\Xip\hookrightarrow L^p(\Omega)^2$ we obtain
		\[
		\lVert v_i \rVert_{L^\infty(\Omega)}
		\le C \lVert v_i\rVert_{W^{1,p}(\Omega)}
		\le C \lVert \nabla v_i\rVert_{L^p(\Omega)}
		\le C \lVert \nabla v_i\rVert_{L^\infty_H L^p_z}.
		\]
Similarly one has 
		\[
		\lVert w_i\rVert_{L^\infty(\Omega)}
		\le C\lVert \text{div}_H v_i \lVert_{L^\infty_H L^p_z}
		\le C\lVert \nabla v_i \rVert_{L^\infty_H L^p_z}.
		\]
This allows us to obtain (ii) via Claim~\ref{SemigroupEstimates} and \ref{SemigroupEstimatesHard} as well as
		\begin{align*}
		\lVert \nabla e^{tA}\mathbb{P}(u_1\cdot \nabla)v_2\rVert_{L^\infty_H L^p_z}
		\le C t^{-1/2} \lVert (u_1 \cdot\nabla) v_2\rVert_{L^\infty_H L^p_z}
		\le C t^{-1/2} \lVert \nabla v_1\rVert_{L^\infty_H L^p_z} \lVert \nabla v_2\rVert_{L^\infty_H L^p_z}.
		\end{align*}
To prove (i) we proceed analogously as above to obtain
		\[
		\lVert v_1\otimes v_2 \rVert_{L^\infty_H L^p_z}
		\le C \lVert \nabla v_i\rVert_{L^\infty_H L^p_z} \lVert v_j\rVert_{L^\infty_H L^p_z},
		\quad
		\lVert w_1 v_2 \rVert_{L^\infty_H L^p_z}
		\le C \lVert \nabla v_1\rVert_{L^\infty_H L^p_z} \lVert v_2\rVert_{L^\infty_H L^p_z}
		\]
where $\{i,j\}=\{1,2\}$ and since $\text{div}\, u_i=0$ we can write 
		\[
		(u_1\cdot \nabla)v_2
		=\nabla \cdot(u_1\otimes v_2)
		=\nabla_H\cdot (v_1\otimes v_2)
		+ \partial_z (w_1 v_2)
		\]
which allows us to apply Claim~\ref{SemigroupEstimates} and \ref{SemigroupEstimatesHard} yielding
		\begin{align*}
		\lVert e^{tA}\mathbb{P}(u_1\cdot \nabla)v_2\rVert_{L^\infty_H L^p_z}
		&= \lVert e^{tA}\mathbb{P}\nabla \cdot (u_1\otimes v_2)\rVert_{L^\infty_H L^p_z}
		\\
		&\le \lVert  e^{tA}\mathbb{P}\nabla_H \cdot (v_1\otimes v_2)\rVert_{L^\infty_H L^p_z}
		+ \lVert e^{tA}\mathbb{P}\partial_z (w_1 v_2)\rVert_{L^\infty_H L^p_z}
		\\
		&\le C t^{-1/2}\left( 
		\lVert v_1\otimes v_2\rVert_{L^\infty_H L^p_z}+\lVert w_1 v_2\rVert_{L^\infty_H L^p_z}
		\right)	
		\\
		&\le C t^{-1/2}\lVert \nabla v_1\rVert_{L^\infty_H L^p_z}\lVert v_2\rVert_{L^\infty_H L^p_z},
		\end{align*}
and estimate (iii) is obtained analogously via
		\begin{align*}
		\lVert \nabla e^{tA}\mathbb{P}(u_1\cdot \nabla)v_2\rVert_{L^\infty_H L^p_z}
		\le C t^{-1}\left( 
		\lVert v_1\otimes v_2\rVert_{L^\infty_H L^p_z}+\lVert w_1 v_2\rVert_{L^\infty_H L^p_z}
		\right)	
		\le C t^{-1}\lVert \nabla v_1\rVert_{L^\infty_H L^p_z}\lVert v_2\rVert_{L^\infty_H L^p_z}.
		\end{align*}
To prove (iv) we observe that $w_i=0$ on $\Gamma_u\cup \Gamma_b$ implies that
		\[
		\PP \partial_z (w_1 v_2)
		=\partial_z (w_1 v_2)
		=-(\text{div}_H v_1)v_2+w_1\partial_z v_2
		\]
and the right-hand side is further estimated via
		\[
		\lVert (\text{div}_H v_1)v_2\rVert_{L^\infty_H L^p_z} 
		\le C\lVert \nabla v_1\rVert_{L^\infty_H L^p_z} \lVert v_2\rVert_{L^\infty(\Omega)}
		\le C \lVert \nabla v_1\rVert_{L^\infty_H L^p_z} \lVert \nabla v_2 \rVert_{L^\infty_H L^p_z},
		\]
and
		\[
		\lVert w_1\partial_z v_2\rVert_{L^\infty_H L^p_z}
		\le \lVert w_1 \rVert_{L^\infty(\Omega)}\lVert \partial_z v_2\rVert_{L^\infty_H L^p_z}
		\le C\lVert \nabla v_1\rVert_{L^\infty_H L^p_z}\lVert \nabla v_2\rVert_{L^\infty_H L^p_z}.
		\]
Applying Claim~\ref{SemigroupEstimates} then yields that for $\{i,j\}=\{1,2\}$ we have
		\begin{align*}
		\lVert e^{tA}\PP \nabla_H \cdot (v_1\otimes v_2)\rVert_{L^\infty_H L^p_z}
		\le C t^{-1/2}\lVert v_1\otimes v_2\rVert_{L^\infty_H L^p_z}
		\le C t^{-1/2} \lVert \nabla v_i\rVert_{L^\infty_H L^p_z} \lVert v_j\rVert_{L^\infty_H L^p_z}, 
		\end{align*}
as well as
		\begin{align*}		
		\lVert e^{tA}\PP \partial_z (w_1 v_2)\rVert_{L^\infty_H L^p_z}
		&\le C  \lVert \nabla v_1\rVert_{L^\infty_H L^p_z} \lVert \nabla v_2\rVert_{L^\infty_H L^p_z}
		\end{align*}
which implies (iv) and completes the proof.
\end{proof}

It has been proven in \cite{GGHHK17} that the operator $\Apos$ possesses maximal $L^q$-regularity. In \cite{GigaGriesHieberHusseinKashiwabara2017} the authors applied this to develop a solution theory for initial data
		\[
		a\in X_{\gamma}:=(\lpso,D(A_p))_{1-1/q,q}\subset B_{pq}^{2-2/q}(\Omega)^2\cap \lpso
		\]
where $p,q\in (1,\infty)$ satisfy $1/p + 1/q \le 1$. In particular, one has the following result.

\begin{lemma}\label{SmoothData}
	Let $a\in X_\gamma$. Then there exists a unique strong solution to the primitive equations \eqref{eq:PrimitiveEquations} with boundary conditions \eqref{eq:bc} satisfying
	\[
	v\in C([0,\infty);X_\gamma).
	\]
\end{lemma}

This enables a key step in the proof of our main result as it guarantees the existence of smooth reference solutions $v_{\text{ref}}$ to the primitive equations given sufficiently smooth reference data $a_{\text{ref}}$. 
In order to construct $v$ as a solution to problem \eqref{eq:PrimitiveEquations} with initial data $a$ we construct $V:=v-v_{\text{ref}}$ by an iterative method using initial data $a_0:=a-a_{\text{ref}}$.
Before we do so, we establish an auxiliary lemma.

\begin{lemma}\label{RecursiveLemma}
Let $(a_n)_{n\in\N}$ be a sequence of positive real numbers such that 
		\[
		a_{m+1}\le a_0+c_1 a_m^2+c_2 a_m \quad \text{ for all } m\in \N 
		\]
and constants $c_1>0$ and $c_2\in (0,1)$ such that $4c_1 a_0<(1-c_2)^2$. Then $a_m<\frac{2}{1-c_2}a_0$ for all $m\in\N$. 
\end{lemma}

\begin{proof}
Let $x_0$ be the smallest solution to the equation $x=a_0+c_1 x^2+c_2x$. Then
		\[
		0<x_0=\frac{(1-c_2)-\sqrt{(1-c_2)^2-4c_1a_0}}{2c_1}
		=\frac{1}{2c_1}\frac{4c_1 a_0}{(1-c_2)+\sqrt{(1-c_2)^2-4c_1 a_0}}
		<\frac{2}{1-c_2}a_0,
		\]
and since $p(x)=a_0+c_1 x^2+c_2 x$ is an increasing function on $[0,\infty)$ it follows that $p(x)\le x_0$ for $x\in[0,x_0]$. The condition $c_2\in(0,1)$ further yields
		\[
		(1-c_2)+\sqrt{(1-c_2)^2-4c_1a_0}<2
		\]
from which it follows that $a_0<x_0$ and thus the claim is easily derived by induction.
\end{proof}
		
We now prove our main result.

\begin{proof}[Proof of Theorem~\ref{MainTheorem}]
\textbf{Step 1:} \textit{Decomposition of data.}

Given an initial value $a\in \Xipos$ we will split it into a smooth part $a_{\text{ref}}$ and a small rough part $a_0$, where $a=a_{\text{ref}}+a_0$, as follows:
Since $\Aipos$ is densely defined on $\Xipos$ we take $a_{\text{ref}}\in D(\Aipos)$ such that
 $a_0:=a-a_{\text{ref}}$ can be assumed to be arbitrarily small in $\Xipos$.
Now let $q\in (1,\infty)$ be such that $1/q+1/p\le 1$ and $2/q+3/p<1$.  The latter condition on $q$ then yields the embedding $X_\gamma \hookrightarrow C^1(\overline{\Omega})^2$. 
Due to $D(\Aipos)\subset D(\Apos)\subset X_{\gamma}$ it follows from Lemma \ref{SmoothData} that taking $a_{\text{ref}}$ as initial data of the primitive equations, there exists a function $v_{\text{ref}}\in C([0,\infty);X_\gamma)$ solving the primitive equations with initial data $v_{\text{ref}}(0)=a_{\text{ref}}$.

\textbf{Step 2:} \textit{Estimates for the construction of a local solution.}

We will show that there exists a constant $C_0>0$ such that if 
$a_0 \in \Xipos$ satisfies $\lVert a_0\rVert_{L^\infty_H L^p_z}<C_0$ then there exists a time $T>0$ and a unique function
		\[
		V\in\mathcal{S}(T):=\{
		V\in C([0,T];\Xipos):\lVert \nabla V(t)\rVert_{L^\infty_H L^p_z}
		= o(t^{-1/2})
		\}, 
		\]
where
		\[
		\lVert V\rVert_{\mathcal{S}(T)}
		=\max\left\{
		\sup_{0<t<T}\lVert V(t)\rVert_{L^\infty_H L^p_z},
		\sup_{0<t<T}t^{1/2}\lVert \nabla V(t)\rVert_{L^\infty_H L^p_z}
		\right\}
		\]
such that $v=v_{\text{ref}}+V$ solves problem \eqref{eq:PrimitiveEquations} on 
$(0,T)$ with initial value $v(0)=a$. In order to construct $V$ we define the iterative sequence of functions $(V_m)_{m\in\N}$ via
		\begin{align}
		V_0(t)=e^{tA}a_0, \quad V_{m+1}(t)=e^{tA}a_0+\int_0^t e^{(t-s)A}F_m(s)\,ds
		\end{align}
where 
		\[
		F_m:=-\PP \left( 
		(U_m\cdot\nabla)V_m
		+(U_m\cdot\nabla)v_{\text{ref}}
		+(u_{\text{ref}}\cdot\nabla)V_m 
		\right)
		\]
and $U_m=(V_m,W_m)$, $u_{\text{ref}}=(v_{\text{ref}},w_{\text{ref}})$ with the vertical component $w$ given by the horizontal component $v$ via the relation \eqref{eq:WRelation}. 
%
%
%
%
We will now estimate this sequence in $\mathcal{S}(T)$ for some value $T>0$ to be fixed later on. 
Since $\PP a_0=a_0$ we have
		\[
		\lVert V_0\rVert_{\mathcal{S}(T)}\le C \lVert a_0\rVert_{L^\infty_H L^p_z},
		\quad T\in (0,\infty)
		\]
by Lemma~\ref{SemigroupEstimates}. 
For $m\ge 1$ we will first consider the gradient estimates. We have already estimated the term $\nabla e^{tA}a_0$, whereas for the convolution integrals we have
		\begin{align*}
		\left\lVert \int_0^{t/2}\nabla e^{(t-s)A}\PP \left(
		(U_m(s)\cdot \nabla) V_m(s)
		\right)\,ds	\right\rVert_{L^\infty_H L^p_z}
		&\le C \left(\int_0^{t/2}(t-s)^{-1}s^{-1/2}\,ds\right) K_m(t) H_m(t)
		\\
		&= C t^{-1/2} K_m(t) H_m(t)
		\end{align*}
by Lemma~\ref{NonlinearEstimates} (iii) where 
		\begin{align*}
		K_m(t):=\sup_{0<s<t} s^{1/2}\lVert \nabla V_m(s)\rVert_{L^\infty_H L^p_z}, 
		\quad
		H_m(t):=\sup_{0<s<t} \lVert V_m(s)\rVert_{L^\infty_H L^p_z}
		\end{align*}
and via Lemma~\ref{NonlinearEstimates} (ii) we obtain
		\begin{align*}
		\left\lVert \int_{t/2}^t\nabla e^{(t-s)A}\PP \left(
		U_m(s)\cdot \nabla V_m(s)
		\right)\,ds\right\rVert_{L^\infty_H L^p_z}
		&\le C \left(
		\int_{t/2}^t(t-s)^{-1/2}s^{-1}
		\right)K_m(t)^2
		\\
		&\le C t^{-1/2} K_m(t)^2.
		\end{align*}
Finally applying Lemma~\ref{NonlinearEstimates} (ii) to the two remaining mixed terms yields
		\begin{align*}
		\left\lVert \int_0^t\nabla e^{(t-s)A}\PP \left(
		U_m(s)\cdot \nabla) v_{\text{ref}}(s)
		\right)\,ds\right\rVert_{L^\infty_H L^p_z}
		&\le C \left(
		\int_0^t (t-s)^{-1/2}s^{-1/2}\,ds
		\right) \sup_{0<s<t}\lVert \nabla v_{\text{ref}}(s)\rVert_{L^\infty_H L^p_z} K_m(t)
		\\
		&=C  \sup_{0<s<t}\lVert \nabla v_{\text{ref}}(s)\rVert_{L^\infty_H L^p_z} K_m(t),
		\\
		\left\lVert \int_0^t\nabla e^{(t-s)A}\PP \left(
		u_{\text{ref}}(s)\cdot \nabla) V_m(s)
		\right)\,ds\right\rVert_{L^\infty_H L^p_z}
		&\le C \sup_{0<s<t}\lVert \nabla v_{\text{ref}}(s)\rVert_{L^\infty_H L^p_z} K_m(t).
		\end{align*}
We set $R:=\sup_{0\le t\le T_0} \lVert \nabla v_{\text{ref}}(t)\rVert_{L^\infty_H L^p_z}$ and note that $0<R<\infty$ by Lemma~\ref{SmoothData}, since $v_{\text{ref}}\in C([0,\infty);X_{\gamma})$ and $2/q+3/p<1$ implies that $X_{\gamma}\subset B^{2-2/q}_{pq}(\Omega)^2\hookrightarrow C^1(\overline{\Omega})^2$ via embedding theory, cf. \cite[Section 3.3.1]{Triebel}.
Taking these estimates together yields
		\begin{align}\label{eq:EstimateWithGradient}
		t^{1/2}\lVert \nabla V_{m+1}(t)\rVert_{L^\infty_H L^p_z} 
		&\le C_1\biggl( \lVert a_0\rVert_{L^\infty_H L^p_z}
		+ K_m(t) H_m(t)
		+ K_m(t)^2
		+ R t^{1/2} K_m(t)\biggr).
		\end{align}
To estimate $\lVert V_{m+1}(t)\rVert_{L^\infty_H L^p_z}$ we apply Lemma~\ref{NonlinearEstimates} (i) to obtain
		\begin{align*}
		\left\lVert \int_0^t e^{(t-s)A}\PP \left(
		(U_m(s)\cdot\nabla)V_m(s)
		\right)\,ds \right\rVert_{L^\infty_H L^p_z}
		&\le C \left(\int_0^t (t-s)^{-1/2}s^{-1/2}\,ds\right) K_m(t) H_m(t)
		\\
		&= C K_m(t) H_m(t)
		\end{align*}
whereas for the mixed terms Lemma~\ref{NonlinearEstimates} (iv) yields
		\begin{align*}
		\lVert e^{(t-s)A}\PP (U_m(s)\cdot\nabla)v_{\text{ref}}(s)\rVert_{L^\infty_H L^p_z}
		\le &C \biggl(
		(t-s)^ {-1/2}\lVert \nabla  v_{\text{ref}}(s)\rVert_{L^\infty_H L^p_z} 
		\lVert V_m(s)\rVert_{L^\infty_H L^p_z}
		\\&+\lVert \nabla V_m(s)\rVert_{L^\infty_H L^p_z} \lVert \nabla v_{\text{ref}}(s)\rVert_{L^\infty_H L^p_z}
		\biggr)
		\end{align*}
and therefore
		\begin{align*}
		\left\lVert
		\int_0^t e^{(t-s)A}\PP ( (U_m(s)\cdot\nabla)v_{\text{ref}}(s))\,ds
		\right\rVert_{L^\infty_H L^p_z}
		&\le C \left(
		\int_0^t (t-s)^{-1/2}\,ds
		\right) R H_m(t)
		+ C \left( \int_0^t s^{-1/2}\,ds\right) R K_m(t)
		\\&= C R t^{1/2} (H_m(t)+K_m(t)),
		\end{align*}
and the other mixed term can be treated analogously due to the symmetry of the right-hand side in (iv). 
Taking these estimates together yields
		\begin{align}\label{eq:EstimateWithoutGradient}
		\lVert V_{m+1}(t)\rVert_{L^\infty_H L^p_z}
		&\le C_1 \biggl(\lVert a_0\rVert_{L^\infty_H L^p_z}
		+ K_m(t) H_m(t)
		+ t^{1/2}H_m(t)
		+ t^{1/2}K_m(t)\biggr).
		\end{align}
Since the right-hand sides of \eqref{eq:EstimateWithGradient} and \eqref{eq:EstimateWithoutGradient} are increasing functions we obtain for $t>0$ that
		\begin{align}\label{Recursive}
		\begin{split}
		K_{m+1}(t)&\le C_1\biggl( 
		\lVert a_0\rVert_{L^\infty_H L^p_z}
		+ K_m(t)H_m(t)
		+ K_m(t)^2
		+ R t^{1/2}K_m(t)
		\biggr),
		\\
		H_{m+1}(t)&\le C_1\biggl( 
		\lVert a_0\rVert_{L^\infty_H L^p_z}
		+ K_m(t)H_m(t)
		+ R t^{1/2}H_m(t)
		+ R t^{1/2}K_m(t)\biggr).
		\end{split}
		\end{align}
Now let $T\in (0,T_0)$ where $T_0>0$ is chosen in such a way that 
		\[
		8C_1 R T_0^{1/2}<1.
		\]
		%
Then for all $0<t\le T<T_0$ we have
		\[
		\lVert V_{m+1}\rVert_{\mathcal{S}(t)}
		\le C_1\lVert a_0\rVert_{L^\infty_H L^p_z}
		+2C_1 \lVert V_{m}\rVert_{\mathcal{S}(t)}^2
		+\frac{1}{4}\lVert V_{m}\rVert_{\mathcal{S}(t)}.
		\]
By Lemma~\ref{RecursiveLemma} it follows that if $8C_1^2\lVert a_0\rVert_{L^\infty_H L^p_z}<(1-1/4)^2$, then for all $m\in \N$ we have
		\begin{align}\label{eq: Vm uniformly bounded}
		\lVert V_{m}\rVert_{\mathcal{S}(t)}
		\le \frac{8}{3}C_1\lVert a_0\rVert_{L^\infty_H L^p_z}, 
		\quad 
		t\in (0,T].
		\end{align}
The property $\lim_{t \to 0+}t^{1/2}\lVert \nabla V_m(t)\rVert_{L^\infty_H L^p_z}=0$ is then easily obtained via induction and Claim~\ref{SemigroupEstimates} (e). 
%
%
%
%

\textbf{Step 3:} \textit{Convergence.}

We now show that $(V_m)_{m\in\N}$ is a Cauchy sequence in $\mathcal{S}(T)$ if $\lVert a_0\rVert_{L^\infty_H L^p_z}$ is sufficiently small. For this purpose we consider the new sequence
		\[
		\tilde{V}_m:=V_{m+1}-V_m,\quad m\ge 0.
		\]
Using the previous estimates we already know that $\lVert \tilde{V}_0\rVert_{\mathcal{S}(T)}<\infty$. 
To estimate this sequence further we use
		\[
		F_m-F_{m-1}=\left( \tilde{U}_{m-1}\cdot \nabla\right)V_m
		+\left( U_{m-1}\cdot \nabla\right)\tilde{V}_{m-1}
		+\left( \tilde{U}_{m-1}\cdot \nabla\right)v_{\text{ref}}
		+\left( U_{\text{ref}}\cdot \nabla\right)\tilde{V}_{m-1}
		\]
and proceed as above to obtain
		\begin{align}
		\begin{split}
		t^{1/2}\lVert \nabla \tilde{V}_m(t)\rVert_{L^\infty_H L^p_z}
		\le C_2\biggl(&
		2 H_m(t) \tilde{K}_{m-1}(t)
		+2 \tilde{H}_{m-1}(t) K_{m-1}(t)
		+K_m(t)\tilde{K}_{m-1}(t)
		\\&+K_{m-1}(t)\tilde{K}_{m-1}(t)
		+2 R t^{1/2} \tilde{K}_{m-1}
		\biggr)
		\end{split}
		\end{align}
as well as
		\begin{align}
		\lVert \tilde{V}_m(t)\rVert_{L^\infty_H L^p_z}
		&\le C_2 \biggl(
		\tilde{K}_{m-1}(t)\left[H_m(t)+H_{m-1}(t)\right]
		+ Rt^{1/2}\left[\tilde{K}_{m-1}(t)+\tilde{H}_{m-1}(t)\right]
		\biggr),
		\end{align}
where
		\[
		\tilde{K}_m(t):=\sup_{0<s<t}s^{1/2}\lVert \nabla \tilde{V}_m(t)\rVert_{L^\infty_H L^p_z},
		\quad
		\tilde{H}_m(t):=\sup_{0<s<t}\lVert \tilde{V}_m(t)\rVert_{L^\infty_H L^p_z}. 
		\]
By \eqref{eq: Vm uniformly bounded} it follows that if 
		\[
		\max\{2 R T_0^{1/2},16 C_1\lVert a_0\rVert_{L^\infty_H L^p_z}\}<1/4C_2,
		\]
then for $m\ge 1$ and $0<t\le T<T_0$ we have
		\[
		\lVert \tilde{V}_m(t)\rVert_{\mathcal{S}(t)}
		\le C_2\biggl(16C_1 \lVert a_0\rVert_{L^\infty_H L^p_z}+2Rt^{1/2}\biggr)
		\lVert \tilde{V}_{m-1}(t)\rVert_{\mathcal{S}(t)}
		<\frac{1}{2}\lVert \tilde{V}_{m-1}(t)\rVert_{\mathcal{S}(t)}.
		\]
Therefore, since $\mathcal{S}(T)$ is a Banach space, $(V_m)_{m\in\N}$ converges in $\mathcal{S}(T)$. We denote the limit by $V$ and see that it satisfies
		\begin{align}\label{eq: mild solution V}
		V(t)=e^{tA}a_0-\int_0^t e^{(t-s)A}\PP\biggl(
		(U(s)\cdot\nabla)V(s)
		+(U(s)\cdot\nabla)v_{\text{ref}}(s)
		+(u_{\text{ref}}(s)\cdot \nabla)V(s)
		\biggr)\,ds
		\end{align}
for $t\in (0,T)$ and thus $v:=V+v_{\text{ref}}$ is a solution to the primitive equations \eqref{eq:PrimitiveEquations}. 
%
%
%
%

\textbf{Step 4:} \textit{Extending to a global solution.}

Using $V\in \mathcal{S}(T)$, the embedding $L^\infty_H L^p_z(\Omega) \hookrightarrow L^p(\Omega)$, as well as the semigroup estimates
		\begin{align*}
		t^{\vartheta}\lVert e^{tA}\PP f\rVert_{D((-\Apos)^{\vartheta})}
		\le C \lVert f\rVert_{L^p(\Omega)}, \quad
		t^{1/2}\lVert e^{tA}\PP \nabla\cdot f\rVert_{L^p(\Omega)}
		\le C \lVert f\rVert_{L^p(\Omega)}, 
		\quad t>0, 
		\quad \vartheta\in [0,1]
		\end{align*}
compare \cite[Lemma 4.6]{HieberKashiwabara2015} and  \cite[Theorem 3.7]{GGHHK17}, one easily obtains that $V(t_0)\in D((-\Apos)^{\vartheta})$ for $t_0>0$, and thus $v(t_0)\in D((-\Apos)^{1/p})$ as well, so $v$ can be extended to a global solution that is strong on $(t_0,\infty)$. 
%
%
%
%

\textbf{Step 5:} \textit{Uniqueness.}

To see that $v$ is a unique solution and thus strong on $(0,t_0)$ as well, we consider $v^{(1)}$ and $v^{(2)}$ both to be solutions in the sense of Theorem~\ref{MainTheorem} with initial value $a$ and set 
		\[
		t^*:=\inf\{t\in[0,\infty):v^{(1)}(t)\neq v^{(2)}(t)\}.
		\]
		%
Assume that $t^*\in (0,\infty)$. Then using continuity of the solutions
		\[
		a^*:=v^{(1)}(t^*)=v^{(2)}(t^*)=a^*_{\text{ref}}+a_0^*
		\]
where $a_0^*\in \Xipos$ is sufficiently small and $a^*_{\text{ref}}\in D(\Aipos)$. Let $v^*_{\text{ref}}$ be the reference solution to the initial data $a^*_{\text{ref}}$ and
		\[
		V^{(i)}(t):=v^{(i)}(t^*+t)-v^*_{\text{ref}}(t^*),
		\quad i=1,2. 
		\]
Then $V^{(1)},V^{(2)}\in \mathcal{S}(T^*)$ both satisfy the condition \eqref{eq: mild solution V} for arbitrary $t\in (0,T^*)$, $T^*\in (0,\infty)$. We set 
$\tilde{V}:=V^{(1)}-V^{(2)}$ and observe that proceeding analogously as before one obtains 
		\begin{align*}
		\tilde{H}(t)
		&\le C_3\biggl(
		t^{1/2}(\tilde{H}(t)
		+\tilde{K}(t))+H^{(1)}(t) \tilde{K}(t)
		+K^{(2)}(t) \tilde{H}(t)
		\biggr),
		\\
		\tilde{K}(t)
		&\le C_3\biggl(
		2t^{1/2}\tilde{K}(t)
		+H^{(1)}(t)\tilde{K}(t)
		+K^{(2)}(t)\tilde{H}(t)
		+K^{(1)}\tilde{K}(t)
		+K^{(2)}\tilde{K}(t)
		\biggr),
		\end{align*}
where $\tilde{H}, H^{(i)}, \tilde{K}, K^{(i)}$ are defined analogously to above. This yields
		\begin{align}\label{eq:UniquenessEstimate}
		\lVert \tilde{V}\rVert_{\mathcal{S}(t)}
		\le C_3 \left(
		t^{1/2}+H^{(1)}(t)+H^{(2)}(t)+K^{(1)}(t)+K^{(2)}(t)
		\right)
		\lVert \tilde{V}\rVert_{\mathcal{S}(t)}, 
		\quad t\in (0,T^*).
		\end{align}
By taking $T^*>0$ to be small the terms $(T^*)^{1/2}$ and $K^{(1)}(T)^*,K^{(2)}(T^*)$ can be taken to be arbitrarily small due to $\lVert \nabla V^{(i)}(t)\rVert=o(t^{-1/2})$, which in the case $t^*=0$ follows from the regularity of $v$ and in the case $t^*>0$ this follows from $\lVert \nabla v(t^*)\in L^\infty_H L^p_z(\Omega)^2$.

As for $H^{(1)}$ and $H^{(2)}$, using the same arguments that derived \eqref{eq:EstimateWithoutGradient} one obtains for $t\in (0,T^*)$ that
		\begin{align}\label{eq:Second Uniqueness Estimate}
		\begin{split}
		H^{(i)}(t)
		& \le C_1 \left( 
		\lVert a^*_0\rVert_{L^\infty_H L^p_z}
		+ K^{(i)}(t) H^{(i)}(t)
		+ R^* t^{1/2} H^{(i)}(t)
		+ R^* t^{1/2} K^{(i)}(t)
		\right),
		\end{split}
		\end{align}
		where $R^*:=\sup_{0\le t\le T^*} \lVert \nabla v^*_{\text{ref}}(t)\rVert_{L^\infty_H L^p_z}$.  
Now, we choose $T\in (0,T^*)$ so small that
		\begin{align*}
		K^{(i)}(T) H^{(i)}(T^*)
		+ R^* T^{1/2} H^{(i)}(T^*)
		+ R^* T^{1/2} K^{(i)}(T^*)
		\le \lVert a^*_0\rVert_{L^\infty_H L^p_z}.
		\end{align*}
Now, taking $\lVert a^*_0\rVert_{L^\infty_H L^p_z}$ to be sufficiently small, using that the constants $C_i>0$, $i=1,2,3$, are independent of $\lVert a^*_0\rVert_{L^\infty_H L^p_z}$, we obtain that the pre-factor in \eqref{eq:UniquenessEstimate} is smaller $1$.
%
%
Hence, it follows that $\lVert \tilde{V}\rVert_{\mathcal{S}(t)}=0$ for $t\in (0,T)$ and thus $v^{(1)}=v^{(2)}$ on $[0,t^*+T)$ which is a contradiction. 

\textbf{Step 6:} Additional regularity.

By \cite[Theorem 6.1]{HieberKashiwabara2015} we thus have 
		\[
		v\in C^1((0,\infty);\lpso)\cap C((0,\infty);W^{2,p}(\Omega))^2,
		\quad
		\pi\in C((0,\infty;W^{1,p}(G)).
		\]
The additional regularity $v\in C([0,\infty];\Xipos)$ follows from the strong continuity of the semigroup on $\Xipos$. 

%
%

For the pressure we have $\pi(t) \in W^{1,p}(G)\hookrightarrow C^{0,\alpha}([0,1]^2)$ for $\alpha \in (0,1-2/p)$. To obtain the regularity of $\nabla_H \pi$, observe that
		\[
		\nabla_H \pi
		=-Bv
		-(1-\mathbb{P})(u\cdot \nabla)v
		=-Bv
		-(1-Q)\overline{(u\cdot \nabla)v}
		\]
where we used that $(1-\mathbb{P})f=(1-Q)\overline{f}$. In the proof of Claim~\ref{SemigroupEstimates} we have already proven that $Bv(t)\in C^{0,\alpha}([0,1]^2)$ for $\alpha \in (0,1-3/p)$ if $v(t)\in W^{2,p}(\Omega)^2$.  Likewise, since $1-Q$ is continuous on $C^{0,\alpha}_{\text{per}}([0,1]^2)^2$ and $v\in C((0,\infty);W^{2,p}(\Omega))^2$, we obtain that the remaining terms belong to $C((0,\infty);C^{0,\alpha}([0,1]^2))^2$.
\end{proof}

\begin{proof}[Proof of Theorem~\ref{PerturbationTheorem}]
Here, we make use of the fact that the relevant estimates in Claim ~\ref{SemigroupEstimates} and Claim~\ref{SemigroupEstimatesHard} can also be applied in $L^{\infty}_HL^p_z(\Omega)^2$, compare Remark~\ref{rem:main}~$(c)$.
	
Let $a=a_1+a_2$ be as in Theorem~\ref{PerturbationTheorem}. Next, we introduce a decomposition setting  $a_0:=a_2+(a_1-a_{\text{ref}})$ where
\begin{align*}
a_{\text{ref}} \in D(\Aipos), 
\quad 
a_1 \in \Xipos 
\quad \hbox{and} \quad 
a_2\in L^{\infty}_HL^p_z(\Omega)^2\cap\lpso,
\end{align*}
where $a_{\text{ref}}$ is such that $a_0$ satisfies the smallness condition of Theorem~\ref{MainTheorem}.


Then the same iteration scheme as in the previous proof can be used to construct $V$ for the initial value $a_0$ and, in turn, $v$ to the initial value $a$. 

The property 
		\[
		v\in C([0,\infty);\lpso)
		\cap L^{\infty}((0,T); L^\infty_H L^p_z(\Omega))^2
		\]
follows from the boundedness and exponential stability of the semigroup on $\lpso$ and $L^\infty_H L^p_z(\Omega)^2\cap \lpso$, as well as the strong continuity on $\lpso$. 
Since the solution regularizes at $t_0>0$, compare Step 4 in the previous proof, we further obtain
$v\in C((0,\infty);\Xipos)$ from the strong continuity on $\Xipos$.	

The condition
		\begin{align}\label{eq: lim sup gradient}
		\limsup_{t\to 0+}t^{1/2}\lVert \nabla v \rVert_{L^\infty_H L^p_z(\Omega)}
		\le C \lVert a_2\rVert_{L^\infty_H L^p_z},
		\end{align}
is verified as follows.
Since Claim~\ref{SemigroupEstimates} yields
		\[
		\limsup_{t\to 0+}t^{1/2}\lVert \nabla e^{tA}(a_1-a_{\text{ref}})\rVert_{L^\infty_H L^p_z}
		=0, \quad 
		t^{1/2}\lVert \nabla e^{tA}a_2\rVert_{L^\infty_H L^p_z}
		\le C_4 \lVert a_2\rVert_{L^\infty_H L^p_z}, 
		\quad t>0,
		\]
one obtains 
		\[
		\limsup_{t\to 0+} t^{1/2}\lVert \nabla V_0(t)\rVert_{L^\infty_H L^p_z}
		\le C_4 \lVert a_2\rVert_{L^\infty_H L^p_z}. 
		\]
We now prove 
$\limsup_{t\to 0+}t^{1/2}\lVert \nabla V_{m}(t)\rVert_{L^\infty_H L^p_z}
\le 2 C_4 \lVert a_2\rVert_{L^\infty_H L^p_z}$ by induction. Assuming the claim holds for $m\in \N$ we obtain
		\begin{align*}
		\limsup_{t \searrow 0}t^{1/2}\lVert \nabla V_{m+1}(t)\rVert_{L^\infty_H L^p_z}
		\le \left(
		1
		+2C_1 \lVert a_0\rVert_{L^\infty_H L^p_z}
		+4C_4 \lVert a_2\rVert_{L^\infty_H L^p_z}
		\right)
		C_4 \lVert a_2\rVert_{L^\infty_H L^p_z}
		\end{align*}
in the same manner as \eqref{eq:EstimateWithGradient}. 
Assuming that $\lVert a_0\rVert_{L^\infty_H L^p_z}<1/4C_1$ and $\lVert a_2\rVert_{L^\infty_H L^p_z}<1/8C_4$ it follows that the claim holds for all $m\in\N$ and by taking the limit the same estimate holds for $V$. Using $v_{\text{ref}}\in C([0,\infty);C^1(\overline{\Omega})^2)$, we obtain that $v=V+v_{\text{ref}}$ satisfies \eqref{eq: lim sup gradient}.

To prove uniqueness we make the following modifications. If $v^{(1)}$ and $v^{(2)}$ are both solutions in the sense of Theorem~\ref{PerturbationTheorem}, we again define $t^*:=\inf\{t\in [0,\infty):v^{(1)}(t)\neq v^{(2)}\}$. 

In the case $t^*>0$ we have $a^*=v^{(1)}(t^*)=v^{(2)}(t^*)\in D((-\Apos)^{\vartheta})$ for any $\vartheta \in[0,1]$, compare Step 4 of the proof of Theorem~\ref{MainTheorem}. Choosing $2\vartheta-3/p>0$ we have that $ D((-\Apos)^{\vartheta})\hookrightarrow \Xipos$ and thus we can decompose $a^*=a^*_{\text{ref}}+a^*_0$ as before and the same argument applies.

If we instead have $t^*=0$ we continue to use the decomposition $a=a_{\text{ref}}+a_0$ where $a_0=a_2+(a_1-a_{\text{ref}})$. In this case we have $\lim_{t\to 0+} K^{(i)}(t)\le C \lVert a_2\rVert_{L^\infty_H L^p_z}$ for an absolute constant $C>0$ and thus the quantities on the right-hand side of \eqref{eq:UniquenessEstimate} can again be taken to be sufficiently small, where on the right-hand side of \eqref{eq:Second Uniqueness Estimate} one has $\lVert a_0\rVert_{L^\infty_H L^p_z}$ instead of $\lVert a^*_0\rVert_{L^\infty_H L^p_z}$, which again yields uniqueness.

This completes the proof.
\end{proof}

\end{document}